\documentclass[10pt,a4paper]{article}

\usepackage{amsmath,amsthm,amssymb}
\usepackage{calc}
\usepackage{bm} 

\newtheorem{thm}{\hspace*{20pt}Theorem}
\newtheorem{lem}{\hspace*{20pt}Lemma}[section]
\newtheorem{pro}{\hspace*{20pt}Proposition} 

\newcounter{constant}[section] 
\newcommand{\const}{\ifnum \theconstant >  0 \stepcounter{constant}\theconstant
                                             \else \setcounter{constant}{1}\theconstant \fi }
\newcounter{subconstant} 
\newcounter{submioneconstant} 
\newcounter{subconstanta}

\newcounter{submitwoconstant}

\newcommand{\suba}{\thesubconstanta} 
\newcommand{\insa}{\setcounter{subconstanta}{\value{constant}}}  
\newcounter{subconstantb} 
\newcommand{\subb}{\thesubconstantb} 
\newcommand{\insb}{\setcounter{subconstantb}{\value{constant}}}  
\newcounter{subconstantc} 
\newcommand{\subc}{\thesubconstantc} 
\newcommand{\insc}{\setcounter{subconstantc}{\value{constant}}}  
\newcounter{subconstantd} 
\newcommand{\subd}{\thesubconstantd} 
\newcommand{\insd}{\setcounter{subconstantd}{\value{constant}}}  
\newcounter{subconstante} 
\newcommand{\sube}{\thesubconstante} 
\newcommand{\inse}{\setcounter{subconstante}{\value{constant}}}  
\newcounter{subconstantf} 
\newcommand{\subf}{\thesubconstantf} 
\newcommand{\insf}{\setcounter{subconstantf}{\value{constant}}}  
\newcounter{subconstantg} 
\newcommand{\subg}{\thesubconstantg} 
\newcommand{\insg}{\setcounter{subconstantg}{\value{constant}}}  
\newcounter{subconstanth} 
\newcommand{\subh}{\thesubconstanth} 
\newcommand{\insh}{\setcounter{subconstanth}{\value{constant}}}  
\newcounter{subconstanti} 
\newcommand{\subi}{\thesubconstanti} 
\newcommand{\insi}{\setcounter{subconstanti}{\value{constant}}}  
\newcounter{subconstantj} 
\newcommand{\subj}{\thesubconstantj} 
\newcommand{\insj}{\setcounter{subconstantj}{\value{constant}}}  
\newcounter{subconstantk} 
\newcommand{\subk}{\thesubconstantk} 
\newcommand{\insk}{\setcounter{subconstantk}{\value{constant}}}  
\newcounter{subconstantl} 
\newcommand{\subl}{\thesubconstantl} 
\newcommand{\insl}{\setcounter{subconstantl}{\value{constant}}}  
\newcounter{subconstantm} 
\newcommand{\subm}{\thesubconstantm} 
\newcommand{\insm}{\setcounter{subconstantm}{\value{constant}}}  
\newcounter{subconstantn} 
\newcommand{\subn}{\thesubconstantn} 
\newcommand{\insn}{\setcounter{subconstantn}{\value{constant}}}  
\newcounter{subconstanto} 
\newcommand{\subo}{\thesubconstanto} 
\newcommand{\inso}{\setcounter{subconstanto}{\value{constant}}}  
\newcounter{subconstantp}

\newcounter{subconstantq}

\newcounter{subconstantr}

\newcounter{subconstants}

\newcounter{subconstantt}

\newcounter{subconstantu}

\newcounter{subconstantv}

\newcounter{subconstantw}

\newcounter{subconstantx}

\newcounter{subconstanty}

\newcounter{subconstantz}

\newcommand{\amin}{\alpha_{\min}}
\newcommand{\binf}{\beta_{\inf}}
\newcommand{\hbinf}{\hat{\beta}_{\inf}}
\newcommand{\bd}{\beta_d}
\newcommand{\supp}{{\rm supp}}
\newcommand{\Iinf}{I_{\inf}}
\newcommand{\tpsi}{\tilde{\psi}}
\newcommand{\psid}{\psi_d}
\newcommand{\hpsi}{\hat{\psi}}
\newcommand{\wka}{w_{k,\alpha}}
\newcommand{\twka}{\tilde{w}_{k,\alpha}}
\newcommand{\twkaone}{\tilde{w}_{k,\alpha}^{(1)}}
\newcommand{\twktwo}{\tilde{w}_{k}^{(2)}}
\newcommand{\twtwo}{\tilde{w}^{(2)}}
\newcommand{\tftwo}{\tilde{f}^{(2)}}
\newcommand{\wk}{w_k}
\newcommand{\tw}{\tilde{w}}
\newcommand{\tf}{\tilde{f}}
\newcommand{\tfk}{\tilde{f}_k}
\newcommand{\twk}{\tilde{w}_k}
\newcommand{\wkak}{w_{k,\alpha_k}}
\newcommand{\wkao}{w_{k,\alpha_0}}
\newcommand{\wkb}{w_{k,\beta}}
\newcommand{\wkamin}{w_{k,\alpha_{\min}}}
\newcommand{\zka}{z_{k,\alpha}}
\newcommand{\zkb}{z_{k,\beta}}
\newcommand{\hka}{h_{k,\alpha}}
\newcommand{\twkb}{\tilde{w}_{k,\beta}}
\newcommand{\hwkb}{\hat{w}_{k,\beta}}
\newcommand{\wkt}{w_{k,\tau}}

\newcommand{\twkttwo}{\tilde{w}_{k,\tau}^{(2)}}
\newcommand{\hw}{\hat{w}}
\newcommand{\hwk}{\hat{w}_k}
\newcommand{\hwb}{\hat{w}_\beta}
\newcommand{\sigk}{\sigma_k}
\newcommand{\sigka}{\sigma_{k,\alpha}}
\newcommand{\sigkao}{\sigma_{k,\alpha_0}}
\newcommand{\muk}{\mu_k}
\newcommand{\tauk}{\tau_k}
\newcommand{\xik}{\xi_k}
\newcommand{\txik}{\tilde{\xi}_k}
\newcommand{\txiktwo}{\tilde{\xi}_{k}^{(2)}}
\newcommand{\hxik}{\hat{\xi}_k}

\newcommand{\hzeta}{\hat{\zeta}}
\newcommand{\Pka}{P_{k,\alpha}}
\newcommand{\Qka}{Q_{k,\alpha}}
\newcommand{\tdk}{\tilde{\delta}_k}
\newcommand{\tgam}{\tilde{\gamma}}
\newcommand{\tgamk}{\tilde{\gamma}_k}
\newcommand{\hgam}{\hat{\gamma}}
\newcommand{\hf}{\hat{f}}
\newcommand{\hfk}{\hat{f}_k}
\newcommand{\hdk}{\hat{\delta}_k}
\newcommand{\hz}{\hat{z}}
\newcommand{\calB}{{\cal B}}
\newcommand{\hcalB}{\hat{\cal B}}
\newcommand{\hI}{\hat{I}}
\newcommand{\hd}{\hat{d}}
\newcommand{\hcalI}{\hat{\cal I}}
\newcommand{\hcalC}{\hat{\cal C}}
\newcommand{\hs}{\hat{s}}
\newcommand{\hc}{\hat{c}}
\newcommand{\hb}{\hat{\beta}}
\newcommand{\hJ}{\hat{J}}
\newcommand{\calP}{{\cal P}}

\makeatletter
\@addtoreset{equation}{section}

\makeatother

\makeatletter
\@addtoreset{equation}{section}

\makeatother

\begin{document}

\title{Extremal boundedness of a variational functional in point vortex mean field theory associated with probability measures}

\author{Takashi Suzuki\footnotemark[1] \ \ \ Ryo Takahashi \footnotemark[2] \ \ \ Xiao Zhang \footnotemark[3]}

\footnotetext[1]{
Division of Mathematical Science, Department of Systems Innovation, 
Graduate School of Engineering Science, Osaka University, 
Machikaneyamacho 1-3, Toyonakashi, 560-8531, Japan. 
(E-mail: {\it suzuki@sigmath.es.osaka-u.ac.jp})
}
\footnotetext[2]{
Division of Mathematical Science, Department of Systems Innovation, 
Graduate School of Engineering Science, Osaka University, 
Machikaneyamacho 1-3, Toyonakashi, 560-8531, Japan. 
(E-mail: {\it r-takaha@sigmath.es.osaka-u.ac.jp})
}
\footnotetext[3]{
Division of Mathematical Science, Department of Systems Innovation, 
Graduate School of Engineering Science, Osaka University, 
Machikaneyamacho 1-3, Toyonakashi, 560-8531, Japan. 
(E-mail: {\it zhangx@sigmath.es.osaka-u.ac.jp})
}

\date{\today}

\maketitle

\begin{abstract}
We study a variational functional of Trudinger-Moser type associated with one-sided Borel probability measure. 
Its boundedness at the extremal parameter holds when the residual vanishing occurs. 
In the proof we use a variant of the Y.Y. Li estimate. 
\end{abstract}

\section{Introduction}

The purpose of the present paper is to study the boundedness of a variational function concerning the  mean field limit of many point vortices \cite{sawada-suzuki08}.  This limit takes the form 
\begin{equation}
-\Delta v=\lambda\int_I \alpha\left(\frac{e^{\alpha v}}{\int_\Omega e^{\alpha v}}-\frac{1}{|\Omega|}\right)\calP(d\alpha) \quad \mbox{on $\Omega$}, \quad \int_\Omega v=0, 
 \label{eqn:SawadaSuzukiModel}
\end{equation}
where $\Omega=(\Omega,g)$ is a compact and orientable Riemannian surface without boundary in dimension two, $dx$ a volume element on $\Omega$, and $|\Omega|$ the volume of $\Omega$: $|\Omega|=\int_\Omega dx$. 
The unknown variable $v$ stands for the stream function of the fluid and $\calP=\calP(d\alpha)$ is a Borel probability measure on $I=[-1,1]$, representing a deterministic distribution of the circulation of vortices. 

Single circulation is described by $\calP=\delta_{+1}$. 
In this simplest case, equation \eqref{eqn:SawadaSuzukiModel} is sometimes called the mean field equation. 
Since Onsager's pioneering work of statistical mechanics on two-dimensional equilibrium turbulence \cite{onsager}, there are numerous mathematical and physical references in this case (see, for instance, \cite{Suzuki07, Tarantello08} and the references therein). 
Also, the other model $\calP=(1-\tau)\delta_{-1}+\tau\delta_{+1}$, $0<\tau<1$, is concerned with signed vortices \cite{jm73, pl76}. Equation \eqref{eqn:SawadaSuzukiModel} is thus regarded as a generalization of these cases.  

There, the deterministic distribution of circulations is described by $P(d\alpha)$. 
Several works are already devoted to equation (\ref{eqn:SawadaSuzukiModel}), particularly, when $\calP$ is atomic \cite{ew09, jwyz08, ors10, os06}, i.e., 
\begin{equation} 
\calP=\sum_{i=1}^N b_i\delta_{\alpha_i}, \quad \alpha_i\in I, \ b_i>0, \ \sum_{i=1}^N b_i=1. 
 \label{eqn:1.2}
\end{equation}
Actually, this model is equivalent to the Liouville system studied by \cite{ck94, csw97, sw05}. 
We note that L. Onsager himself arrived at (\ref{eqn:SawadaSuzukiModel}) for (\ref{eqn:1.2}), see \cite{es06}. 

Model \eqref{eqn:SawadaSuzukiModel} is the Euler-Lagrange equation of the functional 
\begin{equation*}
J_\lambda(v)=\frac{1}{2}\int_\Omega |\nabla v|^2-\lambda\int_I \log\left(\int_\Omega e^{\alpha v}\right)\calP(d\alpha), \quad v\in E, 
\end{equation*}
where 
\[E=\{v\in H^1(\Omega) \ | \ \int_\Omega v=0\}. \]
Hence it may be the first step to clarify its boundedness to study (\ref{eqn:SawadaSuzukiModel}) for general $\calP(d\alpha)$.  It is a kind of the Trudinger-Moser inequality, 
\begin{equation} 
\inf_{v\in E}J_\lambda(v)>-\infty. 
 \label{eqn:1.3}
\end{equation} 

In the atomic case of (\ref{eqn:1.2}), the best constant of $\lambda$ for (\ref{eqn:1.3}) is known \cite{sw05}.  This result is originally described in the dual form of logarithmic HLS inequality (see \cite{ors10} for (\ref{eqn:1.3})). Taking the limit, we can detect the extremal parameter $\lambda=\bar{\lambda}$ for (\ref{eqn:1.3}) to hold \cite{rs-pre}, that is, 
\begin{equation}
\bar{\lambda}=\inf\left\{\left.\frac{8\pi\calP(K_\pm)}{\left(\int_{K_\pm}\alpha \calP(d\alpha)\right)^2} \ \right| \ K\pm \subset I_\pm \cap \supp \ \calP \right\},  
 \label{eqn:optimalconst}
\end{equation}
where
\[ I_+=[0,1], \quad I_-=[-1,0] \]
and 
\[ \supp \ \calP=\{\alpha\in I \mid \mbox{$\calP(N)>0$ for any open neighborhood $N$ of $\alpha$} \}. \]
Thus we obtain 
\begin{align*}
&\lambda<\bar{\lambda} \quad \Rightarrow \quad \inf_{v\in E} J_\lambda(v)>-\infty  \\
&\lambda>\bar{\lambda} \quad \Rightarrow \quad \inf_{v\in E} J_\lambda(v)=-\infty. 
\end{align*}

The inequality 
\begin{equation} 
\inf_{v\in E} J_{\bar{\lambda}}(v)>-\infty 
 \label{eqn:1.6h}
\end{equation} 
however, is open, although (\ref{eqn:1.6h}) is the case if $\calP$ is atomic \cite{sw05}.  Here we take the fundamental assumption 
\begin{equation} 
\supp \ \calP\subset I_+ 
 \label{eqn:1.6}
\end{equation} 
and approach the problem as follows. 

Namely, given $\lambda_k\uparrow\bar{\lambda}$, we have a minimizer $v_k\in E$ of $J_\lambda$, that is,  
\begin{equation*}
\inf_{v\in E}J_{\lambda_k}(v)=J_{\lambda_k}(v_k). 
\end{equation*}
If 
\begin{equation} 
\limsup_{k\rightarrow\infty}J_{\lambda_k}(v_k)>-\infty
 \label{eqn:1.9h}
\end{equation} 
we have 
\[
J_{\bar{\lambda}}(v)=\lim_{k\rightarrow\infty}J_{\lambda_k}(v)\geq\limsup_{k\rightarrow\infty}J_{\lambda_k}(v_k)>-\infty, \quad v\in E, 
\] 
and hence (\ref{eqn:1.6h}) follows. 
In the case of 
\begin{equation*} 
\sup_k\Vert v_k\Vert_\infty<+\infty, 
\end{equation*} 
inequality (\ref{eqn:1.9h}) is valid since $v=v_k$ is a solution to (\ref{eqn:SawadaSuzukiModel}) for $\lambda=\lambda_k$. 

Assuming the contrary, we use a result of \cite{ors10} concerning the non-compact solution sequence $\{(\lambda_k,v_k)\}$ to (\ref{eqn:SawadaSuzukiModel}). Regarding (\ref{eqn:1.6}), we obtain 
\begin{align*} 
&{\cal S}\equiv \{x_0\in\Omega \mid \mbox{there exists $x_k\in\Omega$ such that} \\ 
&\qquad\qquad\qquad\quad \mbox{$x_k\rightarrow x_0$ and $v_k(x_k)\rightarrow+\infty$}\}\neq\emptyset. 
\end{align*}
This blowup set ${\cal S}$ is finite and there is $0\leq s \in L^1(\Omega)\cap L_{loc}^\infty(\Omega\setminus{\cal S})$ such that 
\begin{equation} 
\nu_{k}\equiv\lambda_k\int_{I_+}\frac{\alpha e^{\alpha v_k}}{\int_\Omega e^{\alpha v_k}}\calP(d\alpha) \overset{\ast}{\rightharpoonup} \nu\equiv s+\sum_{x_0\in{\cal S}}n(x_0)\delta_{x_0} \quad \mbox{in ${\cal M}(\Omega)$} 
 \label{eqn:concentration}
\end{equation}
with $n(x_0)\geq 4\pi$ for each $x_0\in{\cal S}$, where $\delta_{x_0}$ denotes the Dirac measure centered at $x_0$ and ${\cal M}(\Omega)$ is the space of measures identified with the dual space of $C(\Omega)$.  Under these preparations our main result is stated as follows. 

\begin{thm}
Inequality (\ref{eqn:1.6h}) holds under the conditions (\ref{eqn:1.6}) and $s=0$ in (\ref{eqn:concentration}). 
 \label{thm:main}
\end{thm}

Henceforth, we put 
\begin{equation}
\amin=\inf_{\alpha \in \supp \ \calP}\alpha.
 \label{eqn:amin}
\end{equation}

Here we have a note concerning the assumption made in the above theorem, that is, $s=0$ in (\ref{eqn:concentration}), 
which we call the {\it residual vanishing}. 
This condition is actually satisfied under a suitable assumption on $\calP$. 
\begin{pro}[\cite{sz-rims} Theorem 3] 
Residual vanishing, $s=0$, occurs to the above $\{ v_k\}$ if $\alpha_{min}>1/2$. 
 \label{pro:2}
\end{pro} 
An immediate consequence is the following theorem. 
\begin{thm}
We have (\ref{eqn:1.6h}) under $\alpha_{min}>1/2$. 
 \label{cor:2}
\end{thm}
So far, there is no known inequality (\ref{eqn:1.6h}) for continuous $\calP(d\alpha)$ except for Theorem \ref{cor:2}. 

Residual vanishing implies the following property used in the proof of Theorem \ref{thm:main}.  
\begin{pro}[\cite{sz-rims} Lemma 3]
If the residual vanishing occurs to the above $\{ v_k\}$ then it holds that 
\begin{equation} 
\sharp{\cal S}\leq 1,\quad \bar{\lambda}=\frac{8\pi}{\left(\int_{I_+} \alpha\calP(d\alpha)\right)^2}. 
 \label{eqn:1.13}
\end{equation} 
 \label{pro:4}
\end{pro} 

Theorem \ref{cor:2} contains the classical case of the Trudinger-Moser inequality for $\calP=\delta_{+1}$. 
We shall show a variant of Y.Y. Li's estimate \cite{yyl99}, which is the key of the proof of Theorem \ref{thm:main}. 
As we see later on, it takes the form that is weaker than the estimate shown for $\calP=\delta_{+1}$ in \cite{yyl99}. 

To state the result, let 
\begin{equation} 
\wka(x)=\alpha v_k(x)-\log\int_\Omega e^{\alpha v_k}, \quad k\in{\bf N}, \ \alpha\in I_+\setminus\{0\},  
 \label{eqn:1.13h}
\end{equation} 
which satisfies 
\begin{equation} 
-\Delta \wka=\alpha\lambda_k\int_{I_+}\beta\left( e^{\wkb}-\frac{1}{\vert\Omega\vert}\right)\calP(d\beta) \quad \mbox{on $\Omega$}, \quad \int_\Omega e^{\wka}=1. 
 \label{eqn:eqnw}
\end{equation} 
Regarding (\ref{eqn:1.13}), we put ${\cal S}=\{x_0\}$. There exists $x_k\in\Omega$ such that 
\[ x_k\rightarrow x_0, \quad v_k(x_k)=\max_\Omega v_k, \quad \wka(x_k)=\max_\Omega \wka. \]
Here we take an isothermal chart $(U_k,\Psi_k)$ satisfying 
\[ \Psi_k(x_k)=0\in{\bf R}^2, \quad g=e^{\xi_k(X)}(dX_1^2+dX_2^2), \quad \xi_k(0)=0. \] 
Then it holds that 
\begin{equation}
-\Delta_X \wka=e^{\xi_k}\left(\alpha\lambda_k \int_{I_+} \beta\left(e^{\wkb}-\frac{1}{|\Omega|}\right)\calP(d\beta)\right) \quad \mbox{in $\Psi_k(U_k)$}, 
 \label{eqn:eqnw'}
\end{equation}
where 
\[ \wka(X)=\wka\circ\Psi^{-1}_k(X). \]

Henceforth, we shall write $X$ by the same notation $x$ for simplicity. 
Also, we do not distinguish any sequences with their subsequences. 
Under this agreement we have $x_k=x_0=0$. 
Moreover, there exists $R_0>0$ such that $B_{3R_0}\subset\subset\Psi_k(U_k)$ and $0$ is the maximizer of $\wka$ in $\bar{B}_{3R_0}$. 

The estimate is now stated in the following proposition. 

\begin{pro} 
Under the assumptions of Theorem \ref{thm:main} it holds that 
\begin{equation} 
\wka(x)-\wka(0)=-\alpha\left(\gamma_0+o(1)\right)\log(1+e^{w_{k,1}(0)/2}\vert x\vert)+O(1) 
 \label{eqn:1.16}
\end{equation} 
as $k\rightarrow\infty$ uniformly in $x\in B_{R_0}$ and $\alpha\in I_+\setminus\{0\}$, where
\[
\gamma_0=\frac{4}{\int_{I_+}\beta \calP(d\beta)}. 
\]
 \label{pro:yyl-type-estimate}
\end{pro}

To conclude this section, we shall describe a sketch of the proof of Theorem \ref{thm:main}. 
The first step is the blowup analysis. 
We put
\[
\twka(x)=\wka(\sigk x)+2\log\sigk,\quad 
\sigk=e^{-\wk(0)/2}\rightarrow 0,\quad
\twk(x)=\tw_{k,1}(x), 
\]
and get 
\begin{align*}
&-\Delta\twka=\alpha (\tfk-\tdk e^{\txik}), \quad \twka\leq \twka(0)\leq\twk(0)=0 \quad \mbox{in $B_{R_0/\sigk}$} \\
&\int_{B_{R_0/\sigk}}e^{\twka+\tilde{\xi}_k}\leq 1, \quad \int_{B_{R_0/\sigk}}\tfk\leq\lambda_k\int_{I_+}\beta\calP(d\beta)
\end{align*} 
for each $\alpha\in I_+\setminus\{0\}$, where 
\[
\tfk=\lambda_k\int_{I_+}\beta e^{\twkb+\txik}\calP(d\beta),\quad 
\tdk=\frac{\sigk^2\lambda_k\int_{I_+}\beta\calP(d\beta)}{\vert \Omega\vert},\quad 
\txik(x)=\xi(\sigk x). 
\]
The compactness argument assures the existence of $\tw=\tw(x)$ and $\tf=\tf(x)$ such that 
\[
\twk\rightarrow\tw,\quad \tilde f_k\rightarrow\tf \quad \mbox{in $C^2_{loc}({\bf R}^2)$} 
\]
and 
\begin{align*}
&-\Delta\tw=\tf\not\equiv 0,\quad \tw\leq\tw(0)=0,\quad 0\leq \tf\leq\bar{\lambda}\int_{I_+}\beta\calP(d\beta) \quad \mbox{in ${\bf R}^2$} \\
&\int_{{\bf R}^2}e^{\tw}\leq 1, \quad \int_{{\bf R}^2}\tf\leq \bar{\lambda}\int_{I_+}\beta\calP(d\beta).
\end{align*}
Next we focus on the quantity
\[
\tgam=\frac{1}{2\pi}\int_{{\bf R}^2}\tf. 
\]
Given a bounded open set $\omega\subset{\bf R}^2$, we have 
\[
\int_{I_+}\left(\int_\omega e^{\twkb+\txik}dx\right)\calP(d\beta)\leq 1, 
\]
and hence there exists $\tilde{\zeta}^\omega=\tilde{\zeta}^\omega(d\beta)\in {\cal M}(I_+)$ such that 
\[
\left(\int_\omega e^{\twkb+\tilde{\xi}_n}dx\right)\calP(d\beta)
\overset{\ast}{\rightharpoonup}\tilde{\zeta}^\omega(d\beta) \quad \mbox{in ${\cal M}(I_+)$}. 
\]
For this limit measure, we can show the absolute continuity with respect to $\calP$, equivalently, the existence of $\tpsi^\omega\in L^1(I_+,\calP)$ such that 
$0\leq \tpsi^\omega\leq 1$ $\calP$-a.e. on $I_+$ and 
\[
\tilde{\zeta}^\omega(\eta)=\int_\eta \tpsi^\omega(\beta)\calP(d\beta)
\]
for any Borel set $\eta\subset I_+$. 
Taking $R_j\uparrow +\infty$ and putting $\omega_j=B_{R_j}$, we have $\tilde{\zeta}\in {\cal M}(I_+)$ and $\tpsi\in L^1(I_+,\calP)$ such that 
\begin{align*}
&0\leq \tpsi(\beta)\leq 1, \quad \mbox{$\calP$-a.e. $\beta$} \\
&0\leq \tpsi^{\omega_1}(\beta) \leq \tpsi^{\omega_2}(\beta) \leq \cdots \rightarrow \tpsi(\beta), \quad \mbox{$\calP$-a.e. $\beta$} \\
&\tilde{\zeta}(\eta)=\int_\eta \tpsi(\beta)\calP(d\beta) \quad\mbox{for any Borel set $\eta\subset I_+$}
\end{align*}
by the monotonicity of $\tpsi^\omega$ with respect to $\omega$. 
Since it holds that 
\[
\bar{\lambda}\int_{I_+}\beta \tpsi^{\omega_j}(\beta)\calP(d\beta)
=\lim_{k\rightarrow\infty}\lambda_k\int_{I_+}\beta\left( \int_{\omega_j}e^{\twkb+\txik}dx\right)\calP(d\beta)
=\int_{\omega_{j}}\tilde f, 
\] 
we have 
\[
\tgam=\frac{1}{2\pi}\int_{{\bf R}^2}\tilde f=\frac{\bar{\lambda}}{2\pi}\int_{I_+}\beta\tpsi(\beta) \ \calP(d\beta)
 \label{eqn:2-1-intro}
\]
by the monotone convergence theorem. 
Furthermore, we use the Pohozaev identity and the behavior of $\tilde{w}$ at infinity to obtain 
\begin{equation}
-\pi\tgam^2=-2\bar{\lambda}\int_{I_+}\tpsi(\beta)\calP(d\beta). 
 \label{eqn:gamma-intro}
\end{equation}
More precisely, the following property holds. 
\begin{pro} 
It holds that 
\begin{equation} 
\tgam=\frac{4}{\int_{I_+}\beta\calP(d\beta)}, 
 \label{eqn:2.17}
\end{equation}
in other words, 
\begin{equation}
\tpsi=1\quad \mbox{$\calP$-a.e. on $I_+$.}
 \label{eqn:2.17'}
\end{equation}
 \label{pro:7}
\end{pro} 
Note that (\ref{eqn:2.17}) and (\ref{eqn:2.17'}) are equivalent by (\ref{eqn:gamma-intro}) and (\ref{eqn:1.13}). 
By virtue of Proposition \ref{pro:7}, 
we can apply the method of \cite{csl07} to prove Proposition \ref{pro:yyl-type-estimate}. 
Finally, we use Proposition \ref{pro:yyl-type-estimate} and the another representation of $J_{\lambda_k}(v_k)$ denoted by 
\begin{align*} 
J_{\lambda_k}(v_k)&=\frac{\lambda_k}{2}\left\{\int_{I_+}(\bar{w}_{k,\alpha}+\wka(0))\calP(d\alpha) \right. \\ 
&\quad \left. +\int_{I_+}\calP(d\alpha)\int_\Omega (\wka(x)-\wka(0))e^{\wka(x)}dx\right\} 
\end{align*} 
to show (\ref{eqn:1.9h}), where $\bar{w}_{k,\alpha}=\frac{1}{|\Omega|}\int_\Omega \wka$. \\ 

This paper consists of five sections and Appendix. 
Sections \ref{sec:mass} and \ref{sec:mass-continued} are devoted to the preliminary and the proof of Proposition \ref{pro:7}, respectively. 
Then, we prove Proposition \ref{pro:yyl-type-estimate} in Section \ref{sec:yyl}. 
The proof of Theorem \ref{thm:main} is provided in Section \ref{sec:proof}. 
An auxiliary lemma in Section \ref{sec:mass} is shown in Appendix. 

\section{Preliminaries}\label{sec:mass}

We start with the following monotonicity properties. 

\begin{lem} 
For $\alpha\in I_+$, we have 
\begin{align}
&\frac{d}{d\alpha}\wka(0)\geq 0
 \label{eqn:monotone-0} \\ 
&\frac{d}{d\alpha}\int_\Omega e^{\alpha v_k}\geq 0. 
 \label{eqn:monotone-1}
\end{align}
\end{lem}

\begin{proof}
We calculate 
\begin{align*}
\frac{d}{d\alpha}\wka(0)
&=v_k(0)-\frac{\int_\Omega v_k e^{\alpha v_k}}{\int_\Omega e^{\alpha v_k}}
\geq v_k(0)-\frac{\int_{\{v_k>0\}} v_k e^{\alpha v_k}}{\int_\Omega e^{\alpha v_k}} \\
&\geq v_k(0)\left(1-\frac{\int_{\{v_k>0\}} e^{\alpha v_k}}{\int_\Omega e^{\alpha v_k}}\right)
\geq 0
\end{align*}
for $k$ and $\alpha\in I_+$, recalling that $0$ is the maximizer of $v_k$, and 
\[
\frac{d}{d\alpha}\int_\Omega e^{\alpha v_k}=\int_\Omega v_k e^{\alpha v_k}=\int_\Omega v_k(e^{\alpha v_k}-1)\geq 0
\]
by using $\int_\Omega v_k=0$ and $s(e^{\alpha s}-1)\geq 0$ which is true for $s\in{\bf R}$ and $\alpha\geq 0$. 
\end{proof}

Henceforth, we put 
\[
w_k(x)=w_{k,1}(x). 
\]
It follows from (\ref{eqn:monotone-0}) that 
\begin{equation} 
w_{k,1}(0)=\max_{\alpha\in I_+} \wka(0). 
 \label{eqn:1.14}
\end{equation} 
The following lemma is the starting point of our blowup analysis. 

\begin{lem}
For every $\alpha\in I_+\setminus\{0\}$, it holds that 
\[ \wka(0)=\max_{\overline{B_{3R_0}}}\wka\rightarrow +\infty. \] 
 \label{lem:2-1}
\end{lem}

\begin{proof} 
Since 
$e^{\wka(0)}=e^{\alpha v_k(0)}/\int_\Omega e^{\alpha v_k}\geq |\Omega|^{\alpha-1} e^{\alpha w_{k,1}(0)}$ for $\alpha\in I_+\setminus\{0\}$, 
it suffices to show that $w_k(0)=w_{k,1}(0)\rightarrow +\infty$. 

Residual vanishing and (\ref{eqn:monotone-1}) imply $\int_\Omega e^{v_k} \rightarrow +\infty$, 
and then the local uniform boundedness of $v_k$ in $\Omega\setminus\{x_0\}$ shows that $w_k\rightarrow -\infty$ locally uniformly in $\Omega\setminus\{x_0\}$. 
Hence if $\lim_{k\rightarrow\infty} w_k(0)=\lim_{k\rightarrow\infty} \|w_k\|_{L^\infty(\Omega)}<+\infty$ then $\lim_{k\rightarrow\infty}\int_\Omega e^{w_k}=0$, 
which contradicts $\int_\Omega e^{w_k}=1$. 
\end{proof}

We put 
\[
\twka(x)=\wka(\sigk x)+2\log\sigk,\quad 
\sigk=e^{-\wk(0)/2}\rightarrow 0,\quad
\twk(x)=\tw_{k,1}(x). 
\]
The last notation is consistent under the agreement of $\wk=w_{k,1}$. 
For each $\alpha\in I_+\setminus \{0\}$ we have 
\begin{align}
&-\Delta\twka=\alpha (\tfk-\tdk e^{\txik}), \quad \twka\leq \twka(0)\leq\twk(0)=0 \quad \mbox{in $B_{R_0/\sigk}$}
 \label{eqn:scale-eqn}\\
&\int_{B_{R_0/\sigk}}e^{\twka+\tilde{\xi}_k}\leq 1, \quad \int_{B_{R_0/\sigk}}\tfk\leq\lambda_k\int_{I_+}\beta\calP(d\beta), 
 \label{eqn:2.5} 
\end{align}
where 
\begin{equation}
\tfk=\lambda_k\int_{I_+}\beta e^{\twkb+\txik}\calP(d\beta),\quad 
\tdk=\frac{\sigk^2\lambda_k\int_{I_+}\beta\calP(d\beta)}{\vert \Omega\vert},\quad \txik(x)=\xi(\sigk x).  
 \label{eqn:def-tilde-f_n}
\end{equation}

We shall use a fundamental fact of which proof is provided in Appendix. 
\begin{lem}
Given $f\in L^1\cap L^\infty({\bf R}^2)$, let 
\[ z(x)=\frac{1}{2\pi}\int_{{\bf R}^2} f(y)\log\frac{|x-y|}{1+|y|}dy. \] 
Then, it holds that 
\[ \lim_{\vert x\vert\rightarrow+\infty}\frac{z(x)}{\log|x|}=\gamma\equiv \frac{1}{2\pi}\int_{{\bf R}^2}f. \]
 \label{lem:2-a}
\end{lem}
The following lemma is also classical (see \cite{pw} p. 130).  
\begin{lem}
If $\phi=\phi(x)$ is a harmonic function on the whole space ${\bf R}^2$ such that 
\begin{equation*}
\phi(x)\leq C_{\const\insb}(1+\log|x|), \quad x\in {\bf R}^2\setminus B_1
\end{equation*}
then it is a constant function. 
 \label{lem:2-b}
\end{lem}

Now we derive the limit of (\ref{eqn:scale-eqn})-(\ref{eqn:2.5}). 

\begin{pro}
It holds that  
\begin{equation}
\twk\rightarrow\tw,\quad \tilde f_k\rightarrow\tf \quad \mbox{in $C^2_{loc}({\bf R}^2)$} 
 \label{eqn:entire-conv}
\end{equation}
for $\tw=\tw(x)$ and $\tf=\tf(x)$ satisfying 
\begin{align}
&-\Delta\tw=\tf\not\equiv 0, \quad \tw\leq\tw(0)=0,\quad 0\leq \tf\leq\bar{\lambda}\int_{I_+}\beta\calP(d\beta) \quad \mbox{in ${\bf R}^2$} \nonumber\\
&\int_{{\bf R}^2}e^{\tw}\leq 1, \quad \int_{{\bf R}^2}\tf\leq \bar{\lambda}\int_{I_+}\beta\calP(d\beta).
 \label{eqn:entire-eqn}
\end{align}
In addition, it holds that 
\begin{equation}
\tw(x)\geq -\tgam\log(1+|x|)+\frac{1}{2\pi}\int_{{\bf R}^2}\tf(y)\log\frac{|y|}{1+|y|}dy
 \label{eqn:entire-estimate-below}
\end{equation}
for any $x\in{\bf R}^2$, where 
\begin{equation}
\tgam=\frac{1}{2\pi}\int_{{\bf R}^2}\tf. 
 \label{eqn:entire-mass}
\end{equation}
 \label{pro:mass}
\end{pro}

\begin{proof} 
We have 
\begin{equation}
\twkb(x)=\beta\twk(x)+(\wkb(0)-\wk(0))
 \label{eqn:relation-w}
\end{equation}
for any $\beta\in I_+\setminus\{0\}$, and also 
\begin{equation}
\twk\leq\twk(0)=0, \quad \wkb(0)\leq\wk(0), \quad \beta\in I_+\setminus\{0\}
 \label{eqn:2.12}
\end{equation} 
by (\ref{eqn:1.14}). 
Hence $\tilde f_k=\tilde f_k(x)$ satisfies 
\begin{equation} 
0\leq \tfk(x)\leq\lambda_k\int_{I_+}\beta\calP(d\beta)\cdot \sup_{B_{R_0}}e^{\xi_k} \quad \mbox{in $B_{R_0/\sigk}$}. 
 \label{eqn:estimate-tfk}
\end{equation}

Fix $L>0$ and decompose $\twk$, $k\gg 1$, as $\twk=\tw_{1,k}+\tw_{2,k}+\tw_{3,k}$, where $\tw_{k,j}$, $j=1,2,3$, are the solutions to 
\begin{align*}
&-\Delta\tw_{1,k}=\tfk\geq 0 \quad \mbox{in $B_L$}, \quad \quad \tw_{1,k}=1 \quad \mbox{on $\partial B_L$} \\
&-\Delta\tw_{2,k}=-\tdk e^{\txik}\leq 0 \quad \mbox{in $B_L$}, \quad \tw_{2,k}=0 \quad \mbox{on $\partial B_L$} \\
&-\Delta\tw_{3,k}=0 \quad \mbox{in $B_L$}, \quad \quad \quad \ \tw_{3,k}=\twk-1 \quad \mbox{on $\partial B_L$}. 
\end{align*}

First, there exists $C_{\const\insc,L}>0$ such that  
\[ 1\leq \tw_{1,k}\leq C_{\subc,L} \quad \mbox{on $\overline{B_L}$}. \]
Next it follows from $\tdk\rightarrow 0$ that 
\[ -\frac{1}{2}\leq \tw_{2,k}\leq 0 \quad \mbox{on $\overline{B_L}$}. \]
Finally, we have 
\[ \tw_{3,k}\leq -1 \quad \mbox{on $\overline{B_L}$} \]
by $\tw_{k}\leq 0$.  Hence $\tw_{3,k}=\tw_{3,k}(x)$ is a negative harmonic function in $B_L$.  Then the Harnack inequality yields $C_{\const\insd,L}>0$ such that  
\[ \tw_{3,k} \geq -C_{\subd,L} \quad \mbox{in $\overline{B_{L/2}}$}. \]
We thus end up with 
\begin{equation}
\frac{1}{2}-C_{\subd,L}\leq \twk \leq \twk(0)=0 \quad \mbox{in $B_{L/2}$}, 
 \label{eqn:lem:2-2'}
\end{equation}
and then the standard compactness argument assures the limit (\ref{eqn:entire-conv})-(\ref{eqn:entire-eqn}) 
thanks to (\ref{eqn:estimate-tfk}), (\ref{eqn:lem:2-2'}), $\xi_k(0)=0$ and $\tdk\rightarrow 0$. 

If $\tf\equiv 0$ then 
\[ -\Delta\tw=0,\quad \tw\leq\tw(0)=0\quad \mbox{in ${\bf R}^2$}, \quad \int_{{\bf R}^2}e^{\tw}\leq 1, \]
which is impossible by the Liouville theorem, and hence $\tilde f\not\equiv 0$.

Since $\tf\in L^1\cap L^\infty({\bf R}^2)$, the function 
\begin{equation} 
\tilde{z}(x)=\frac{1}{2\pi}\int_{{\bf R}^2}\tf(y)\log\frac{|x-y|}{1+|y|}dy 
 \label{eqn:2.15}
\end{equation} 
is well-defined, and satisfies
\begin{equation}
\frac{\tilde{z}(x)}{\log|x|}\rightarrow \tgam=\frac{1}{2\pi}\int_{{\bf R}^2}\tf \quad \mbox{as $|x|\rightarrow +\infty$}
 \label{eqn:lem:2-2''}
\end{equation}
by Lemma \ref{lem:2-a}. Also (\ref{eqn:lem:2-2''}) implies 
\begin{align*}
&-\Delta{\tw}=\tf,\quad -\Delta{\tilde{z}}=-\tf,\quad \tw\leq\tw(0)=0\quad \mbox{in ${\bf R}^2$} \\
&\tilde{z}(x)\leq (\tgam+1)\log \vert x\vert, \quad x\in{\bf R}^2\setminus B_R 
\end{align*}
for some $R>0$ by (\ref{eqn:lem:2-2''}). 
Hence we obtain $\tilde{u}\equiv \tw+\tilde{z}\equiv \mbox{constant}$ by Lemma \ref{lem:2-b}. 
Since $\tw(0)=0$ it holds that 
\begin{equation} 
\tw(x)=-\tilde{z}(x)+\tilde{z}(0). 
 \label{eqn:lem:2-2'''}
\end{equation} 

Now we note 
\begin{align*}
\tilde{z}(x)
&\leq \frac{1}{2\pi}\int_{{\bf R}^2}\tf(y)\log\frac{|x|+|y|}{1+|y|}dy \\
&\leq \log(1+|x|)\cdot\frac{1}{2\pi}\int_{{\bf R}^2}\tf
=\tgam\log(1+|x|) 
\end{align*}
by $\tf\geq 0$. 
Hence, $\tw(x)\geq -\tgam\log(1+|x|)+\tilde{z}(0)$, and the proof is complete. 
\end{proof} 

To study $\tgam$ in (\ref{eqn:entire-mass}), let  
\begin{equation} 
\calB=\{ \beta\in\supp \ \calP \mid \limsup_{k\rightarrow\infty}(\wkb(0)-w_{k}(0))>-\infty\}. 
 \label{eqn:2.20h}
\end{equation} 
From the proof of Proposition \ref{pro:mass}, it follows that if $\calP(\calB)=0$ then $\tf\equiv 0$, a contradiction.  
Hence $\calP(\calB)>0$, and the value 
\begin{equation} 
\binf=\inf_{\beta\in\calB}\beta 
 \label{eqn:2.21h}
\end{equation} 
is well-defined. 
Then we find 
\begin{equation}
\calB=\Iinf\cap\supp\calP
 \label{eqn:equal-biinf}
\end{equation}
by the monotonicity (\ref{eqn:monotone-0}), where 
\[
\Iinf=
\begin{cases}
[\binf,1] & \mbox{if $\binf\in\calB$}, \\
(\binf,1] & \mbox{if $\binf\not\in\calB$}.
\end{cases}
\]

\begin{lem}
$\binf\tgam>2$. 
 \label{lem:2-d}
\end{lem}

\begin{proof} 
By the definition,  every $\beta\in\calB$ admits a subsequence such that $\tilde w_{k, \beta}(0)=\wkb(0)-\wk(0)=O(1)$. 
We recall that $\twkb$ satisfies (\ref{eqn:scale-eqn}) for $\alpha=\beta$, i.e., 
\[
-\Delta\twkb=\beta(-\Delta\twk)=\beta(\tfk-\tdk e^{\txik}). 
\] 
From the argument developed for the proof of (\ref{eqn:entire-conv})-(\ref{eqn:entire-estimate-below}), we have $\tw_\beta=\tw_\beta(x)\in C^2({\bf R}^2)$ such that 
\begin{equation}
\twkb\rightarrow\tw_\beta \quad \mbox{in $C^2_{loc}({\bf R}^2)$}. 
 \label{eqn:convergence-twkb}
\end{equation}
The limit $\tilde w_\beta$ satisfies 
\begin{equation*}
-\Delta\tw_\beta=\beta\tf,\ \tw_\beta \leq\tw_\beta(0)\leq 0 \quad \mbox{in ${\bf R}^2$}, 
\quad \int_{{\bf R}^2}e^{\tw_\beta}\leq 1 
\end{equation*}
and 
\begin{equation}
\tw_\beta(x)\geq -\beta\tgam\log(1+|x|)+\frac{\beta}{2\pi}\int_{{\bf R}^2}\tf(y)\log\frac{|y|}{1+|y|}dy
 \label{eqn:2-d-2}
\end{equation}
with $\tf=\tf(x)$ given in Proposition \ref{pro:mass}. 

If $\binf\in \calB$, we take $\beta=\binf$.  
Since $\tf\in L^1\cap L^\infty({\bf R}^2)$ and $\int_{{\bf R}^2}e^{\tilde w_\beta}<+\infty$, we obtain the lemma by (\ref{eqn:2-d-2}).  

If $\binf\not\in \calB$, we take $\beta_j\in \calB$ in $\beta_j\downarrow \binf$ and apply (\ref{eqn:2-d-2}) for $\beta=\beta_j$. 
Since 
\[
\frac{\beta_j}{2\pi}\int_{{\bf R}^2}\tf(y)\log\frac{\vert y\vert}{1+\vert y\vert}dy=O(1), \quad \int_{{\bf R}^2}e^{\tilde w_{\beta_j}}\leq 1 
\]  
there is $\varepsilon_0>0$ independent of $j$ such that $\beta_j\tgam\geq 2+\varepsilon_0$, and then we obtain the lemma. 
\end{proof}

Given a bounded open set $\omega\subset{\bf R}^2$, we have 
\[ \int_{I_+}\left(\int_\omega e^{\twkb+\txik}dx\right)\calP(d\beta)\leq 1. \]
Hence it holds that 
\begin{equation}
\left(\int_\omega e^{\twkb+\tilde{\xi}_n}dx\right)\calP(d\beta)
\overset{\ast}{\rightharpoonup}\tilde{\zeta}^\omega(d\beta) \quad \mbox{in ${\cal M}(I_+)$}. 
 \label{eqn:2-0}
\end{equation}
Now we shall show that the limit measure $\tilde{\zeta}^\omega=\tilde{\zeta}^\omega(d\beta)\in{\cal M}(I_+)$ is absolutely continuous with respect to $\calP$. 

\begin{lem}
There exists $\tpsi^\omega\in L^1(I_+,\calP)$ such that 
$0\leq \tpsi^\omega\leq 1$ $\calP$-a.e. on $I_+$ and 
\[ \tilde{\zeta}^\omega(\eta)=\int_\eta \tpsi^\omega(\beta)\calP(d\beta) \]
for any Borel set $\eta\subset I_+$. 
 \label{lem:2-e}
\end{lem}

\begin{proof} 
Let $\eta\subset I_+$ be a Borel set and $\varepsilon>0$. Then each compact set $K\subset \eta$ admits an open set $J\subset I_+$ such that 
\[ K\subset\eta\subset J, \quad \calP(J)\leq \varepsilon+\calP(K). \]
Now we take $\varphi\in C(I_+)$ satisfying 
\[ \varphi=1 \quad \mbox{on $K$}, \quad 
0\leq\varphi\leq 1 \quad \mbox{on $I_+$}, \quad \supp\varphi\subset J. \]
Then (\ref{eqn:2-0}) implies 
\begin{align*}
&\tilde{\zeta}^\omega(K)=\int_K \tilde{\zeta}^\omega(d\beta)\leq \int_{I_+} \varphi(\beta) \tilde{\zeta}^\omega(d\beta) \\ 
&\quad =\lim_{k\rightarrow\infty}\int_{I_+}\varphi(\beta)
\left(\int_\omega e^{\twkb+\tilde \xi_k}dx\right)\calP(d\beta) \leq \int_{I_+}\varphi(\beta)\calP(d\beta) \\
&\quad \leq \int_J \calP(d\beta)=\calP(J)\leq \varepsilon+\calP(\eta), 
\end{align*}
and therefore 
\[ 0\leq \tilde{\zeta}^\omega(\eta)=\sup\{\tilde{\zeta}^\omega(K) \mid K\subset\eta \ \mbox{:compact}\} \leq\varepsilon+\calP(\eta). \]
This shows the absolute continuity of $\tilde \zeta^\omega$ with respect to $\calP$. 
\end{proof}

We take $R_j\uparrow +\infty$ and put $\omega_j=B_{R_j}$. From the monotonicity of $\tpsi^\omega$ with respect to $\omega$, there exist $\tilde{\zeta}\in {\cal M}(I_+)$ and $\tpsi\in L^1(I_+,\calP)$ such that 
\begin{align*}
&0\leq \tpsi(\beta)\leq 1, \quad \mbox{$\calP$-a.e. $\beta$} \\
&0\leq \tpsi^{\omega_1}(\beta) \leq \tpsi^{\omega_2}(\beta) \leq \cdots \rightarrow \tpsi(\beta), \quad \mbox{$\calP$-a.e. $\beta$} \\
&\tilde{\zeta}(\eta)=\int_\eta \tpsi(\beta)\calP(d\beta) \quad\mbox{for any Borel set $\eta\subset I_+$}. 
\end{align*}
First, (\ref{eqn:entire-conv}) implies 
\[
\bar{\lambda}\int_{I_+}\beta \tpsi^{\omega_j}(\beta)\calP(d\beta)
=\lim_{k\rightarrow\infty}\lambda_k\int_{I_+}\beta\left( \int_{\omega_j}e^{\tilde \wkb+\tilde \xi_k}dx\right)\calP(d\beta)
=\int_{\omega_{j}}\tilde f.
\] 
Then we obtain 
\begin{equation}
\tgam=\frac{1}{2\pi}\int_{{\bf R}^2}\tilde f=\frac{\bar{\lambda}}{2\pi}\int_{I_+}\beta\tpsi(\beta) \ \calP(d\beta)
 \label{eqn:2-1}
\end{equation}
by the monotone convergence theorem. 

Similarly to \cite{cl93}, on the other hand, we have the following lemma, where $(r,\theta)$ denotes the polar coordinate in ${\bf R}^2$. 
\begin{lem}
We have 
\[ \lim_{r\rightarrow+\infty}r\tw_r=-\tgam, \quad \lim_{r\rightarrow+\infty}\tw_\theta=0 \] 
uniformly in  $\theta$. 
 \label{lem:2-f}
\end{lem}

\begin{proof} 
From (\ref{eqn:2.15}) and (\ref{eqn:lem:2-2'''}), it follows that 
\begin{align*}
&r\tw_r(x)=-\tgam-\frac{1}{2\pi}\int_{{\bf R}^2}\frac{y\cdot (x-y)}{|x-y|^2}\tf(y)dy \\
&\tw_\theta(x)=\frac{1}{2\pi}\int_{{\bf R}^2}\frac{\bar{y}\cdot(x-y)}{|x-y|^2}\tf(y)dy,\quad \bar{y}=(y_2,-y_1). 
\end{align*}
Hence it suffices to show 
\[ \lim_{\vert x\vert\rightarrow+\infty}I_1(x)=\lim_{\vert x\vert\rightarrow+\infty}I_2(x)=0, \] 
where 
\[
I_1(x)=\int_{\vert x-y\vert>\vert x\vert/2}\frac{\vert y\vert}{\vert x-y\vert}\tf(y)dy,\quad 
I_2(x)=\int_{\vert x-y\vert\leq \vert x\vert/2}\frac{\vert y\vert}{\vert x-y\vert}\tf(y)dy. 
\]
Since $\tilde f\in L^1({\bf R}^2)$, we have $\lim_{\vert x\vert\rightarrow+\infty}I_1(x)=0$ by the dominated convergence theorem.  

Next, (\ref{eqn:entire-conv}) implies 
\begin{align*}
&I_2(x) =\lim_{k\rightarrow\infty}\int_{\vert x-y\vert \leq \vert x\vert/2}\frac{|y|}{|x-y|}\left(\lambda_k\int_{I_+}\beta e^{\twkb(y)+\txik(y)}\calP(d\beta)\right)dy \\
&\quad =\bar{\lambda}\lim_{k\rightarrow\infty}\int_{[\binf,1]}\beta\left(\int_{\vert x-y\vert\leq \vert x\vert/2}\frac{\vert y\vert}{\vert x-y\vert}e^{\twkb(y)+\txik(y)}dy\right)\calP(d\beta), 
\end{align*}
recalling (\ref{eqn:2.20h})-(\ref{eqn:2.21h}). Now we use (\ref{eqn:relation-w})-(\ref{eqn:2.12}) and (\ref{eqn:entire-conv}) with (\ref{eqn:lem:2-2'''}), to confirm 
\[ \twkb(x)\leq \beta \twk(x)=\beta(-\tilde z(x)+\tilde z(0))+o(1) \] 
as $k\rightarrow\infty$, locally uniformly in $x\in {\bf R}^2$.  
Hence it holds that 
\[
0\leq I_2(x)\leq C_{\const\inse}\int_{\vert x-y\vert \leq \vert x\vert/2}\frac{\vert y\vert}{\vert x-y\vert}\cdot \int_{[\binf, 1]}e^{-\beta\tilde z(y)}P(d\beta) \ dy. 
\] 
Then (\ref{eqn:lem:2-2''}) and Lemma \ref{lem:2-d} imply 
\[ 0\leq I_2(x)\leq C_{\const\inse}\vert x\vert^{-(1+\varepsilon_0)}\int_{\vert x-y\vert\leq \vert x\vert/2}\frac{dy}{\vert x-y\vert}\leq C_{\const\inse}|x|^{-\varepsilon_0} \] 
with some $\varepsilon_0>0$, where we have used  
\[
\vert x-y\vert \leq \frac{\vert x\vert}{2} \quad \Rightarrow \quad \frac{1}{2}\vert x\vert\leq \vert y\vert \leq \frac{3}{2}\vert x\vert. 
\] 
Hence $\lim_{\vert x\vert\rightarrow+\infty}I_2(x)=0$ follows. 
\end{proof}

The Pohozaev identity
\begin{align}
&R\int_{\partial B_R}\frac{1}{2}\vert \nabla u\vert^2-u_r^2 \ ds=R\int_{\partial B_R}A(x)F(u) \ ds \nonumber\\ 
&\quad -\int_{B_R}2A(x)F(u) + F(u)(x\cdot\nabla A(x)) \ dx
 \label{eqn:2.24}
\end{align}
is valid to $u=u(x)\in C^2(\overline{B_R})$ satisfying
\begin{equation}  
-\Delta u=A(x)F'(u) \quad \mbox{in $B_R$}, 
 \label{eqn:2.25}
\end{equation} 
where $F\in C^1({\bf R})$, $A\in C^1(\overline{B_R})$, and $ds$ denotes the surface element on the boundary. 
By this identity and Lemma \ref{lem:2-f} we obtain the following fact.  
\begin{lem}
It holds that 
\begin{equation}
\int_{I_+}\tpsi(\beta)\calP(d\beta)
=\left(\int_{I_+}\phi_0(\beta)\tpsi(\beta)\calP(d\beta)\right)^2,  
 \label{eqn:lem:2-g}
\end{equation}
where 
\begin{equation} 
\phi_0(\beta)=\frac{\beta}{\int_{I_+}\alpha\calP(d\alpha)}. 
 \label{eqn:2.28h}
\end{equation} 
 \label{lem:2-g}
\end{lem}

\begin{proof} 
We apply (\ref{eqn:2.24}) for (\ref{eqn:2.25}) to (\ref{eqn:scale-eqn}), $\alpha=1$, where $u=\tilde \wk$ and 
\[ F(\twk)=\lambda_k\int_{I_+}e^{\twkb}\calP(d\beta)-\tdk\twk, \quad A(x)=e^{\txik(x)}=e^{\xi_k(\sigk x)}. \]
It follows that 
\begin{align}
&R\int_{\partial B_R} \frac{1}{2}\vert \nabla\twk\vert^2-(\twk)_r^2 \ ds =-2\lambda_k\int_{I_+}\left(\int_{B_{R}}e^{\twkb+\txik}dx\right)\calP(d\beta) \nonumber\\ 
&\quad +R\lambda_k\int_{I_+}\left(\int_{\partial B_R}e^{\twkb+\txik}ds\right)\calP(d\beta)-R\tdk\int_{\partial B_R}\twk e^{\txik}ds \nonumber\\
&\quad -\sigk\cdot\lambda_k\int_{I_+}\left(\int_{B_R}e^{\twkb+\txik}(x\cdot\nabla\xi_k(\sigk x))dx\right)\calP(d\beta) \nonumber\\ 
&\quad +\tdk\int_{B_R}2\twk e^{\txik}+\sigk\twk e^{\txik} (x\cdot\nabla\xi_k(\sigk x)) \ dx
 \label{eqn:2-2}
\end{align}
for $\tdk$ defined by (\ref{eqn:def-tilde-f_n}). 

For every $R>0$, the last three terms on the right-hand side of (\ref{eqn:2-2}) vanish as $k\rightarrow\infty$, because of $\tdk\rightarrow 0$ and $\sigk\rightarrow 0$. For the second term we argue similarly as in the proof of Lemma \ref{lem:2-f}, while the conclusion of Lemma \ref{lem:2-f} is applicable to the left-hand side on (\ref{eqn:2-2}).  We thus end up with 
\begin{equation} 
-\pi\tgam^2=-2\bar{\lambda}\int_{I_+}\tpsi(\beta)\calP(d\beta) 
 \label{eqn:2.29}
\end{equation} 
by sending $k\rightarrow\infty$ and then $R\rightarrow+\infty$.  

Combining (\ref{eqn:2.29}), (\ref{eqn:2-1}), and the value $\bar{\lambda}$ given in (\ref{eqn:1.13}), we obtain (\ref{eqn:lem:2-g})-(\ref{eqn:2.28h}). 
\end{proof}

Let 
\begin{align*}
&{\cal I}^0(\psi)=\int_{I_+}\phi_0(\beta)\psi(\beta)\calP(d\beta) \\
&{\cal C}_d=\{\psi \ | \ 0\leq \psi(\beta)\leq 1 \ \mbox{$\calP$-a.e. on $I_+$ and } \int_{I_+}\psi(\beta)\calP(d\beta)=d \}
\end{align*}
and $\chi_A$ be the characteristic function of the set $A$. 
The following lemma is a variant of the result of \cite{lieb-loss}. 

\begin{lem}
For each $0<d\leq 1$, the value $\sup_{\psi\in{\cal C}_d}{\cal I}^0(\psi)$ is attained by 
\begin{equation} 
\psid(\beta)=\chi_{\{\phi_0>s_d\}}(\beta)+c_d\chi_{\{\phi_0=s_d\}}(\beta)
 \label{eqn:2.31h}
\end{equation} 
with $s_d$ and $c_d$ defined by 
\begin{align}
&s_d=\inf\{t \ | \ \calP(\{\phi_0>t\})\leq d\} \nonumber\\
&c_d\calP(\{\phi_0=s_d\})=d-\calP(\{\phi_0>s_d\}),\quad 0\leq c_d \leq 1.  
 \label{eqn:2.31}
\end{align}
Furthermore, the maximizer is unique in the sense that $\psi_m=\psid$ $\calP$-a.e. on $I_+$ for any maximizer $\psi_m\in{\cal C}_d$. 
 \label{lem:2-h}
\end{lem}

\begin{proof}
Fix $0<d\leq 1$. 
Given $\psi\in{\cal C}_d$, we compute 
\begin{align}
&\int_{I_+}\phi_0(\psi_d-\psi)\calP(d\beta)=\int_{\{\phi_0>s_d\}}\phi_0(\psi_d-\psi)\calP(d\beta) \nonumber\\
&\quad +s_d\int_{\{\phi_0=s_d\}}(\psi_d-\psi)\calP(d\beta)
-\int_{\{\phi_0<s_d\}}\phi_0\psi\calP(d\beta) \nonumber\\
&\geq s_d \int_{\{\phi_0>s_d\}}(\psi_d-\psi)\calP(d\beta) \nonumber\\ 
&\quad +s_d\int_{\{\phi_0=s_d\}}(\psi_d-\psi)\calP(d\beta)-\int_{\{\phi_0<s_d\}}\phi_0\psi\calP(d\beta)
 \label{eqn:2-h-1} \\
&\geq s_d \left(\int_{\{\phi_0>s_d\}}(\psi_d-\psi)\calP(d\beta)\right. \nonumber\\
&\quad \left. +\int_{\{\phi_0=s_d\}}(\psi_d-\psi)\calP(d\beta)-\int_{\{\phi_0<s_d\}}\psi\calP(d\beta)\right)
 \label{eqn:2-h-2} \\
&=s_d\int_{I_+}(\psi_d-\psi)\calP(d\beta)=0, \nonumber
\end{align}
which means that $\psi_d$ is the maximizer. 

The equalities hold in (\ref{eqn:2-h-1}) and (\ref{eqn:2-h-2}) if and only if $\psi$ is the maximizer, 
and so we shall derive the conditions that the former is true. 
The first condition is that $(\phi_0-s_d)(\psi_d-\psi)=0$ $\calP$-a.e. on $\{\phi_0>s_d\}$, so that 
\begin{equation}
\psi=\psi_d\quad \mbox{$\calP$-a.e. on $\{\phi_0>s_d\}$}
 \label{eqn:2-h-3}
\end{equation}
by the monotonicity of $\phi_0$ and $\psi_d\geq \psi$ on $\{\phi_0>s_d\}$. 
The second one is that $(s_d-\phi_0)\psi=0$ $\calP$-a.e. on $\{\phi_0<s_d\}$, or 
\begin{equation}
\psi=0\quad \mbox{$\calP$-a.e. on $\{\phi_0<s_d\}$}
 \label{eqn:2-h-4}
\end{equation}
by the monotonicity of $\phi_0$ and $\psi\geq 0$. 
The uniqueness follows from (\ref{eqn:2-h-3})-(\ref{eqn:2-h-4}) and $\psi_d, \psi\in{\cal C}_d$. 
\end{proof}

\section{Proof of Proposition \ref{pro:7}}\label{sec:mass-continued}

For the purpose, we assume the contrary, that is, $\tpsi\in{\cal C}_d$ for some $0<d<1$. 
We shall prove Proposition \ref{pro:7} by contradiction. \\ 

Since $\tpsi=\tpsi(\beta)$ satisfies (\ref{eqn:lem:2-g}), it holds that 
\begin{equation*}
d=\int_{I_+} \tpsi(\beta) \calP(d\beta)=\left(\int_{I_+} \phi_0(\beta)\tpsi(\beta) \calP(d\beta)\right)^2. 
\end{equation*}
Lemma \ref{lem:2-h} and (\ref{eqn:2.28h}) yield
\begin{equation}
d=\calP(\{\phi_0>s_d\})+c_d\calP(\{\phi_0=s_d\})\leq \left(\frac{\int_{I_+}\psid(\beta)\beta\calP(d\beta)}{\int_{I_+}\beta\calP(d\beta)}\right)^2
 \label{eqn:new-b-2}
\end{equation}
for $\psid=\psid(\beta)$ defined by (\ref{eqn:2.31h})-(\ref{eqn:2.31}). 
By the monotonicity of $\phi_0=\phi_0(\beta)$, there exists the unique element $\bd\in I_+$ such that 
\[
\phi_0(\bd)=s_d, 
\]
and then (\ref{eqn:new-b-2}) reads 
\begin{equation}
d=\calP((\bd,1])+c_d\calP(\{\bd\})\leq \left(\frac{\int_{(\bd,1]}\beta\calP(d\beta)+c_d\bd\calP(\{\bd\})}{\int_{I_+}\beta\calP(d\beta)}\right)^2. 
 \label{eqn:new-b-3}
\end{equation}

Here we introduce 
\[
H(\tau)=\calP((\bd,1])+\tau\calP(\{\bd\})-\left(\frac{\int_{(\bd,1]}\beta\calP(d\beta)+\tau\bd\calP(\{\bd\})}{\int_{I_+}\beta\calP(d\beta)}\right)^2. 
\]
It follows from (\ref{eqn:optimalconst}) and (\ref{eqn:1.13}) that 
\begin{equation}
H(0)\geq 0,\quad H(1)\geq 0. 
 \label{eqn:new-b-4}
\end{equation} 
Moreover, we have either $c_d=0$ or $c_d=1$ if $\calP(\{\bd\})>0$. 
In fact, since 
\[
H''(\tau)=const.=-2\left(\frac{\bd\calP(\{\bd\})}{\int_{I_+}\beta\calP(d\beta)}\right)^2<0
\]
by $\calP(\{\bd\})>0$, it holds that $H(\tau)>0$ for $0<\tau<1$ by (\ref{eqn:new-b-4}). 
On the other hand, $H(c_d)\leq 0$ by (\ref{eqn:new-b-3}). 

We now claim 
\begin{equation}
\tpsi=\psid=\chi_{I_d}\quad \mbox{$\calP$-a.e. on $I_+$}, 
 \label{eqn:new-b-claim2}
\end{equation}
where 
\[
I_d=
\begin{cases}
[\bd,1] & \mbox{if $\calP(\{\bd\})=0$ or if $\calP(\{\bd\})>0$ and $c_d=1$} \\
(\bd,1] & \mbox{otherwise (i.e., $\calP(\{\bd\})>0$ and $c_d=0$)}. 
\end{cases}
\]
First, we assume that $\calP(\{\bd\})=0$. 
Then, $H(\tau)=H(0)$ for $\tau\in[0,1]$. 
In this case, the equality holds in (\ref{eqn:new-b-3}) by (\ref{eqn:new-b-4}), and thus 
\[
d=\left(\int_{I_+} \phi_0(\beta)\psid(\beta) \calP(d\beta)\right)^2=\left(\int_{I_+} \phi_0(\beta)\tpsi(\beta) \calP(d\beta)\right)^2, 
\]
which means $\tpsi=\psid$ $\calP$-a.e. on $I_+$ by the uniqueness of Lemma \ref{lem:2-h}. 
Note that the integrands are non-negative. 
It is clear that $\psid=\chi_{I_d}$ $\calP$-a.e. on $I_+$. 
Next we assume that $\calP(\{\bd\})>0$. 
Then, we use (\ref{eqn:new-b-3})-(\ref{eqn:new-b-4}) to obtain $H(c_d)=0$, 
which again implies that the equality holds in (\ref{eqn:new-b-3}), and hence 
\[
\tpsi=\psid=
\begin{cases}
\chi_{[\bd,1]} & \mbox{if $c_d=1$} \\
\chi_{(\bd,1]} & \mbox{if $c_d=0$}. 
\end{cases}
\]
The claim (\ref{eqn:new-b-claim2}) is established. 

Property (\ref{eqn:new-b-claim2}) is actually refined as follows, recall (\ref{eqn:equal-biinf}), i.e., 
\[
\calB=\Iinf\cap\supp\calP. 
\]

\begin{lem}
$\tpsi=\chi_{\Iinf}$ $\calP$-a.e. on $I_+$. 
 \label{lem:new-b}
\end{lem}

\begin{proof} There are the following six possibilities: 
\begin{align*}
&\mbox{(i) $\bd<\binf$}\quad \mbox{(ii) $\bd>\binf$} \\
&\mbox{(iii) $\bd=\binf$, $I_d=(\bd,1]$ and $\binf\in\Iinf$} \\ 
&\mbox{(iv) $\bd=\binf$, $I_d=[\bd,1]$ and $\binf\not\in\Iinf$} \\
&\mbox{(v) $\bd=\binf$, $I_d=[\bd,1]$ and $\binf\in\Iinf$} \\ 
&\mbox{(vi) $\bd=\binf$, $I_d=(\bd,1]$ and $\binf\not\in\Iinf$}
\end{align*}
The lemma is clearly true for the cases (v)-(vi), 
and thus it suffices to prove $\calP(I_d\setminus\Iinf)=0$, $\calP(\Iinf\setminus I_d)=0$ and $\calP(\{\bd\})=\calP(\{\binf\})=0$ for the cases (i), (ii) and (iii)-(iv), respectively. 

(i) Assume $\calP(I_d\setminus\Iinf)>0$. 
Then, 
\begin{equation}
\tpsi(\beta)=0\quad \mbox{for $\beta\in I_d\setminus\Iinf$}
 \label{eqn:new-b-5}
\end{equation}
by the definitions of $\Iinf$ and $\tpsi$. 
Note that $\tilde{w}_{k,\beta}\rightarrow -\infty$ locally uniformly in ${\bf R}^2$ for $\beta\in I_d\setminus\Iinf$. 
On the other hand, $\tpsi(\beta)=1$ for some $\beta\in I_d\setminus\Iinf$ by (\ref{eqn:new-b-claim2}), which contradicts (\ref{eqn:new-b-5}). 

(ii) Assume $\calP(\Iinf\setminus I_d)>0$. 
Then, 
\begin{equation}
\tpsi(\beta)=0\quad \mbox{for $\calP$-a.e. $\beta\in\Iinf\setminus I_d$}
 \label{eqn:new-b-6}
\end{equation}
by (\ref{eqn:new-b-claim2}). 
On the other hand, $\tpsi(\beta)>0$ for any $\beta\in\Iinf\setminus I_d$ by the definitions of $\Iinf$ and $\tpsi$, and by the convergence (\ref{eqn:convergence-twkb}), which contradicts (\ref{eqn:new-b-6}). 

(iii) If $\calP(\{\bd\})=\calP(\{\binf\})>0$ then $\tpsi(\bd)=\tpsi(\binf)=0$ by (\ref{eqn:new-b-claim2}) and $I_d=(\bd,1]$. 
On the other hand, $\tpsi(\bd)=\tpsi(\binf)>0$ by $\binf\in\Iinf$ as shown for the case (ii) above, a contradiction. 

(iv) If $\calP(\{\bd\})=\calP(\{\binf\})>0$ then $\tpsi(\bd)=\tpsi(\binf)=1$ by (\ref{eqn:new-b-claim2}) and $I_d=[\bd,1]$. 
On the other hand, $\tpsi(\bd)=\tpsi(\binf)=0$ by $\binf\not\in\Iinf$ as shown for the case (i) above, a contradiction. 
\end{proof} 

Since the equality holds in (\ref{eqn:new-b-3}) as shown above, it follows from (\ref{eqn:new-b-claim2}) and Lemma \ref{lem:new-b} that 
\begin{equation}
\calP(\Iinf)>0,\quad 
\frac{\calP(\Iinf)}{\left(\int_{\Iinf} \beta\calP(d\beta)\right)^2}=\frac{1}{\left(\int_{I_+} \beta\calP(d\beta)\right)^2}
 \label{eqn:new-d-0}
\end{equation}


\begin{lem}
For every $R>0$ and $\alpha\in I_+\setminus\Iinf$, it holds that 
\[
\lim_{k\rightarrow\infty}\int_{B_{\sigka R}}e^{\wka}=0, 
\]
where $\sigka=e^{-\wka(0)/2}$. 
 \label{lem:new-c}
\end{lem}

\begin{proof} Fix $R>0$ and $\alpha\in I_+\setminus\Iinf$. 
Putting
\[
\twkaone(x)=\wka(\sigka x)+2\log\sigka, 
\]
we have 
\[
\int_{B_{\sigka R}}e^{\wka}=\int_{B_R}e^{\twkaone},\quad \twkaone\leq\twkaone(0)=0\quad \mbox{in $B_R$}, 
\]
and therefore it suffices to show 
\[
\twkaone\rightarrow -\infty\quad \mbox{locally uniformly in $B_R\setminus\{0\}$}. 
\]
If this is not the case, then there exist $C_{\const\inse}>0$ and $r_1>0$ such that 
\[
\max_{\overline{B}_R\setminus B_{r_1}}\twkaone\geq -C_{\sube}
\]
for $k\gg 1$. 
Since there exists $y_k\in \overline{B}_R\setminus B_{r_1}$ such that 
\[
\twkaone(y_k)=\max_{\overline{B}_R\setminus B_{r_1}}\twkaone,
\]
it holds that 
\[
\twkaone(y_k)-\twkaone(0)\geq -C_{\sube}
\]
for $k\gg 1$, and thus  
\begin{equation}
\wk(\sigka y_k)-\wk(0)=(\twkaone(y_k)-\twkaone(0))/\alpha\geq -C_{\sube}/\alpha 
 \label{eqn:new-c-1}
\end{equation}
for $k\gg 1$. 
On the other hand, we have $\beta\in\Iinf$ satisfying
\begin{equation}
\lim_{L\rightarrow+\infty}\lim_{k\rightarrow\infty}\int_{B_{\sigk L}}e^{\wkb}=1
 \label{eqn:new-c-2}
\end{equation}
by the definitions of $\Iinf$ and $\tpsi$, and by the convergence shown in the proof of Lemma \ref{lem:2-d}. 

Now, we introduce 
\[
\twktwo(x)=\wkb(\muk x+\sigka y_k)+2\log\muk,\quad \muk=e^{-\wkb(\sigka y_k)/2}. 
\]
Since $\beta\in\Iinf$, there exists $C_{\const\insf}>0$ such that 
\begin{equation}
\wk(0)\leq \wkb(0)+C_{\subf}
 \label{eqn:new-c-3}
\end{equation}
for $k\gg 1$. 
Moreover, it follows from (\ref{eqn:new-c-1}) and (\ref{eqn:new-c-3}) that 
\begin{equation}
\wk(0)-\wkb(\sigka y_k)\leq C_{\const\insg}
 \label{eqn:new-c-4}
\end{equation}
for $k\gg 1$, where $C_{\subg}=\beta C_{\sube}/\alpha+C_{\subf}$. 
Noting (\ref{eqn:new-c-4}) and 
\[
\wk(0)=\sup_{\alpha\in I_+}\sup_{x\in\Omega}\wka(x), 
\]
we find 
\begin{align}
&-\Delta\twktwo=\beta\lambda_k e^{\txiktwo}\int_{I_+}\tau\left(e^{\twkttwo}-\frac{\muk^2}{|\Omega|}\right)\calP(d\tau)\quad \mbox{in $B_{\frac{r_1}{2\muk}}$}
 \label{eqn:new-c-6}\\
&\twktwo(0)=0,\quad \twkttwo\leq \wk(0)-\wkb(\sigka y_k)\leq C_{\subg}\quad \mbox{in $B_{\frac{r_1}{2\muk}}$ for any $\tau\in I_+$} \nonumber\\
&\int_{B_{\frac{r_1}{2\muk}}} e^{\twkttwo+\txiktwo}\leq 1\quad \mbox{for any $\tau\in I_+$}, \nonumber
\end{align}
where
\[
\twkttwo(x)=\wkt(\muk x+\sigka y_k)+2\log\muk,\quad \txiktwo(x)=\xik(\muk x+\sigka y_k). 
\]
Noting $\sigka y_k\rightarrow 0$, $\xik(0)=0$ and the smoothness of $\xik$, 
we perform the compactness argument, similarly to the proof of Proposition \ref{pro:mass}, 
to obtain $\twtwo,\tftwo\in C^2({\bf R}^2)$ and $C_{\const\insh}$ such that 
\[
\twktwo\rightarrow\twtwo,\quad \mbox{[r.h.s. of (\ref{eqn:new-c-6})]}\rightarrow\tftwo\quad \mbox{in $C_{loc}^2({\bf R}^2)$} 
\]
and 
\begin{align*}
&-\Delta \twtwo=\tftwo,\quad \twtwo\leq C_{\subg}\quad \mbox{in ${\bf R}^2$} \\
&\twtwo(0)=0,\quad \int_{{\bf R}^2}e^{\twtwo}\leq 1,\quad \int_{{\bf R}^2} \tftwo+\|\tftwo\|_{L^\infty({\bf R}^2)}\leq C_{\subh}. 
\end{align*}
Therefore, there exist $\ell_1>0$ and $0<\delta\ll 1$ such that 
\begin{equation}
\int_{B_{\muk\ell_1(\sigka y_k)}} e^{\wkb}\geq 2\delta
 \label{eqn:new-c-9}
\end{equation}
for $k\gg 1$. 

On the other hand, (\ref{eqn:new-c-2}) admits $\ell_2>0$ satisfying 
\begin{equation}
\int_{B_{\sigk\ell_2}}e^{\wkb}\geq 1-\delta
 \label{eqn:new-c-10}
\end{equation}
for $k\gg 1$. 
In addition, 
\[
|\sigka y_k|-\muk\ell_1-\sigk\ell_2
\geq \sigka\left(r_1-\frac{\muk}{\sigka}\ell_1-\frac{\sigk}{\sigka}\ell_2\right)
\geq \frac{\sigka r_1}{2}
\]
for $k\gg 1$ since 
\[
\frac{\sigk}{\sigka}\rightarrow 0,\quad \frac{\muk}{\sigka}\rightarrow 0
\]
by $\alpha\not\in\Iinf$ and (\ref{eqn:new-c-4}). 
Hence there holds 
\begin{equation}
B_{\muk\ell_1(\sigka y_k)}\cap B_{\sigk\ell_2}=\emptyset
 \label{eqn:new-c-11}
\end{equation}
for $k\gg 1$. 
Combining (\ref{eqn:new-c-9})-(\ref{eqn:new-c-11}) shows 
\[
1=\int_\Omega e^{\wkb}\geq \int_{B_{\muk\ell_1(\sigka y_k)}\cup B_{\sigk\ell_2}} e^{\wkb}\geq 2\delta+(1-\delta)=1+\delta>1
\]
for $k\gg 1$, a contradiction. 
\end{proof}

\begin{lem}
There are no $\calP$-measurable sets $K_1, K_2 \subset I_+$ satisfying
\begin{equation}
\begin{cases}
\calP(K_i)>0\ (i=1,2),\quad \calP(K_1\cap K_2)=0 \\
\frac{\calP(K_i)}{\left(\int_{K_i}\beta\calP(d\beta)\right)^2}=\frac{1}{\left(\int_{I_+}\beta\calP(d\beta)\right)^2}\ (i=1,2). 
\end{cases}
 \label{eqn:new-e}
\end{equation}
 \label{lem:new-e}
\end{lem}

\begin{proof} Assume that there exist $\calP$-measurable sets $K_1, K_2 \subset I_+$ satisfying (\ref{eqn:new-e}), and put
\[
a_i=\calP(K_i),\quad b_i=\frac{\int_{K_i}\beta\calP(d\beta)}{\int_{I_+}\beta\calP(d\beta)}, 
\]
so that 
\[
a_i=b_i^2\quad \mbox{($i=1,2$)}. 
\]
On the other hand, (\ref{eqn:optimalconst}) and (\ref{eqn:1.13}) show 
\[
\frac{1}{\left(\int_{I_+}\beta\calP(d\beta)\right)^2}
\leq \frac{\calP(K_1\cup K_2)}{\left(\int_{K_1\cup K_2}\beta\calP(d\beta)\right)^2}
=\frac{\calP(K_1)+\calP(K_2)}{\left(\int_{K_1}\beta\calP(d\beta)+\int_{K_2}\beta\calP(d\beta)\right)^2}
\]
or 
\[
\frac{a_1+a_2}{(b_1+b_2)^2}\geq 1. 
\]
Hence we have 
\[
b_1^2+b_2^2\geq (b_1+b_2)^2, 
\]
which is impossible since $b_i>0$ ($i=1,2$). 
\end{proof}

\begin{lem}
There exists $C_{\const\insa}>0$, independent of $k\gg 1$, such that 
\begin{equation}
\sup_{\alpha\in I_+}\sup_{x\in B_{2R_0}}\{\wka(x)+2\log|x|\}\leq C_{\suba}
 \label{eqn:new-d}
\end{equation}
for $k\gg 1$. 
 \label{lem:new-d}
\end{lem}

\begin{proof} The proof is divided into four steps. \\

{\it Step 1}.\ Assume the contrary, that is, there exist $\alpha_k\in I_+$ and $x_k\in \bar{B}_{R_0}$ such that 
\[
M_k\equiv\wkak(x_k)+2\log|x_k|=\sup_{\alpha\in I_+}\sup_{x\in B_{2R_0}}\{\wka(x)+2\log|x|\}\rightarrow +\infty. 
\]
We have
\begin{align*}
&\wkak(x_k)=M_k-2\log|x_k|\geq M_k-2\log R_0\rightarrow+\infty \\
&\ell_k\equiv e^{\wkak(x_k)/2}\cdot\frac{|x_k|}{2}=\frac{e^{M_k/2}}{2}\rightarrow+\infty. 
\end{align*}
For any $x\in B_{|x_k|/2}(x_k)$, $\alpha\in I_+$ and $k$, it holds that 
\begin{align*}
&\wka(x)-\wkak(x_k) \\
&=(\wka(x)+2\log|x|)-(\wkak(x_k)+2\log|x_k|)+2\log\frac{|x_k|}{|x|}\leq 2\log 2. 
\end{align*}
We put 
\begin{align*}
&\hwk(x)=\wkak(\tauk x+x_k)+2\log\tauk,\quad \tauk=e^{-\wkak(x_k)/2}, 
\end{align*}
and get 
\begin{align}
&-\Delta\hwk=\hfk-\hdk e^{\hxik}\quad \mbox{in $B_{\ell_k}$} \nonumber\\
&\hwkb\leq 2\log 2\quad \mbox{in $B_{\ell_k}$ for any $\beta\in I_+$ and $k$} 
 \label{eqn:new-d-1}\\
&\hwk(0)=0,\quad \int_{B_{\ell_k}}e^{\hwkb+\hxik}\leq 1\quad \mbox{for any $\beta\in I_+$ and $k$}, \nonumber
\end{align}
where 
\begin{align*}
&\hfk(x)=\alpha_k\lambda_k\int_{I_+}\beta e^{\hwkb+\hxik}\calP(d\beta),\quad \hdk=\frac{\alpha_k\lambda_k\tau_k^2\int_{I_+}\beta\calP(d\beta)}{|\Omega|}, \\
&\hwkb(x)=\wkb(\tauk x+x_k)+2\log\tauk,\quad \hxik(x)=\xik(\tauk x+x_k). 
\end{align*}
The compactness argument, similarly to the proof of Proposition \ref{pro:mass}, admits $\hw,\hf\in C^2({\bf R}^2)$ and $C_{\const\insi}>0$ such that 
\begin{equation}
\hwk\rightarrow\hw,\quad \hfk\rightarrow\hf\quad \mbox{in $C_{loc}^2({\bf R}^2)$} 
 \label{eqn:new-d-2}
\end{equation}
and 
\begin{align}
&-\Delta \hw=\hf\not\equiv 0,\quad \hw\leq 2\log 2\quad \mbox{in ${\bf R}^2$}, \nonumber\\
&\hw(0)=0,\quad \int_{{\bf R}^2}e^{\hw}\leq 1,\quad \|\hf\|_{L^\infty({\bf R}^2)}+\int_{{\bf R}^2} \hf\leq C_{\subi}. 
  \label{eqn:new-d-3}
\end{align}
Note that 
\begin{equation}
\alpha_0=\lim_{k\rightarrow\infty} \alpha_k\neq 0
 \label{eqn:new-d-4}
\end{equation}
by the Liouville theorem. 
Since $\hf\in L^1\cap L^\infty({\bf R}^2)$, the function 
\begin{equation} 
\hz(x)=\frac{1}{2\pi}\int_{{\bf R}^2}\hf(y)\log\frac{|x-y|}{1+|y|}dy 
 \label{eqn:new-d-5}
\end{equation} 
is well-defined, and satisfies
\begin{equation}
\frac{\hz(x)}{\log|x|}\rightarrow \hgam=\frac{1}{2\pi}\int_{{\bf R}^2}\hf \quad \mbox{as $|x|\rightarrow +\infty$}
 \label{eqn:new-d-6}
\end{equation}
by Lemma \ref{lem:2-a}. 
Similarly to the proof of Proposition \ref{pro:mass}, we see 
\begin{align}
&\hw(x)=-\hz(x)+\hz(0), 
 \label{eqn:new-d-7} \\
&\hw(x)\geq -\hgam\log(1+|x|)+\frac{1}{2\pi}\int_{{\bf R}^2}\hf(y)\log\frac{|y|}{1+|y|}dy, \nonumber
\end{align}
for any $x\in{\bf R}^2$, where 
\begin{equation*}
\hgam=\frac{1}{2\pi}\int_{{\bf R}^2}\hf. 
\end{equation*}

{\it Step 2}.\ We introduce 
\begin{equation}
\hcalB=\{ \beta\in\supp \ \calP \mid \limsup_{k\rightarrow\infty}(\wkb(x_k)-w_{k,\alpha_k}(x_k))>-\infty\}
 \label{eqn:new-d-10}
\end{equation}
and put 
\begin{equation}
\hbinf=\inf_{\beta\in\hcalB}\beta. 
 \label{eqn:new-d-11}
\end{equation}
Note that $\calP(\hcalB)>0$, and so $\hbinf$ is well-defined, since $\hf\not\equiv 0$ as in (\ref{eqn:new-d-3}). 

In this step, we shall show 
\begin{equation}
\frac{\hbinf\hgam}{\alpha_0}>2, 
 \label{eqn:new-d-12}
\end{equation}
where $\alpha_0$ is as in (\ref{eqn:new-d-4}). 

By the definition of $\hcalB$, for every $\beta\in\hcalB$, there exists a subsequence such that $\hwkb(0)=\wkb(x_k)-\wkak(x_k)=O(1)$. 
It follows from (\ref{eqn:new-d-1}) that 
\[
-\Delta\hwkb=\frac{\beta}{\alpha_k}(\hfk-\hdk e^{\hxik})\quad \mbox{in $B_{\ell_k}$}. 
\]
We repeat the procedure developed in Step 1 to obtain $\hwb=\hwb(x)\in C^2({\bf R}^2)$ satisfying 
\begin{align*}
&\hwkb\rightarrow\hwb\quad \mbox{in $C_{loc}^2({\bf R}^2)$}, \\
&-\Delta\hwb=\frac{\beta}{\alpha_0}\hf,\ \hwb\leq \hwb(0)\leq 0\quad \mbox{in ${\bf R}^2$},\quad \int_{{\bf R}^2} e^{\hwb}\leq 1, 
\end{align*}
and 
\begin{equation}
\hwb(x)\geq -\frac{\beta}{\alpha_0}\hgam\log(1+|x|)+\frac{\beta}{2\pi\alpha_0}\int_{{\bf R}^2}\hf(y)\log\frac{|y|}{1+|y|}dy, 
 \label{eqn:new-d-13}
\end{equation}
where $\hf=\hf(x)$ is the limit function in (\ref{eqn:new-d-2}). 

If $\hbinf\in\hcalB$ then we take $\beta=\hbinf$, and obtain (\ref{eqn:new-d-12}) by (\ref{eqn:new-d-13}) and 
\[
\hf\in L^1\cap L^\infty({\bf R}^2),\quad \int_{{\bf R}^2}e^{\hwb}\leq 1. 
\]
If $\hbinf\not\in\hcalB$ then we take $\beta_j\in \hcalB$ satisfying $\beta_j\downarrow \hbinf$, 
and obtain $\varepsilon_1>0$ independent of $j$ such that $\beta_j\tgam\geq 2+\varepsilon_1$, using (\ref{eqn:new-d-13}) for $\beta=\beta_j$ and 
\[
\frac{\beta_j}{2\pi}\int_{{\bf R}^2}\hf(y)\log\frac{\vert y\vert}{1+\vert y\vert}dy=O(1), \quad \int_{{\bf R}^2}e^{\hw_{\beta_j}}\leq 1, 
\] 
and thus (\ref{eqn:new-d-12}) is shown. 

\ \\ 

{\it Step 3}.\ Given a bounded open set $\omega\subset{\bf R}^2$, it holds that 
\[
\int_{I_+}\left(\int_\omega e^{\hwkb+\hxik}dx\right)\calP(d\beta)\leq 1, 
\]
and hence 
\[
\left(\int_\omega e^{\hwkb+\hxik}dx\right)\calP(d\beta)
\overset{\ast}{\rightharpoonup}\hzeta^\omega(d\beta) \quad \mbox{in ${\cal M}(I_+)$}
\]
for some $\hzeta^\omega \in{\cal M}(I_+)$. 
Similarly to the proof of Lemma \ref{lem:2-e}, we see that 
there exists $\hpsi^\omega\in L^1(I_+,\calP)$ such that 
$0\leq \hpsi^\omega\leq 1$ $\calP$-a.e. on $I_+$ and 
\[
\hzeta^\omega(\eta)=\int_\eta \hpsi^\omega(\beta)\calP(d\beta) 
\]
for any Borel set $\eta\subset I_+$. 
We take $R_j\uparrow +\infty$ and put $\omega_j=B_{R_j}$. 
From the monotonicity of $\hpsi^\omega$ with respect to $\omega$, there exist $\hzeta\in {\cal M}(I_+)$ and $\hpsi\in L^1(I_+,\calP)$ such that 
\begin{align*}
&0\leq \hpsi(\beta)\leq 1, \quad \mbox{$\calP$-a.e. $\beta$} \\
&0\leq \hpsi^{\omega_1}(\beta) \leq \hpsi^{\omega_2}(\beta) \leq \cdots \rightarrow \hpsi(\beta), \quad \mbox{$\calP$-a.e. $\beta$} \\
&\hzeta(\eta)=\int_\eta \hpsi(\beta)\calP(d\beta) \quad\mbox{for any Borel set $\eta\subset I_+$}. 
\end{align*}
It follows from (\ref{eqn:new-d-2}) that 
\[
\alpha_0\bar{\lambda}\int_{I_+}\beta \hpsi^{\omega_j}(\beta)\calP(d\beta)
=\lim_{k\rightarrow\infty}\alpha_k\lambda_k\int_{I_+}\beta\left( \int_{\omega_j}e^{\hwkb+\hxik}dx\right)\calP(d\beta)
=\int_{\omega_{j}}\hf, 
\] 
and thus we obtain 
\begin{equation}
\hgam=\frac{1}{2\pi}\int_{{\bf R}^2}\hf=\frac{\alpha_0\bar{\lambda}}{2\pi}\int_{I_+}\beta\hpsi(\beta) \ \calP(d\beta), 
 \label{eqn:new-d-14}
\end{equation}
sending $j\rightarrow \infty$. 

Complying the proof of Lemma \ref{lem:2-f}, one can show that  
\begin{equation}
\lim_{r\rightarrow+\infty}r\hw_r=-\hgam, \quad \lim_{r\rightarrow+\infty}\hw_\theta=0, 
 \label{eqn:new-d-15}
\end{equation}
by using (\ref{eqn:new-d-5}), (\ref{eqn:new-d-7}), (\ref{eqn:new-d-2}), (\ref{eqn:new-d-10})-(\ref{eqn:new-d-11}), (\ref{eqn:new-d-6}), (\ref{eqn:new-d-12}) 
and the property, derived from (\ref{eqn:new-d-2}), (\ref{eqn:new-d-7}) and (\ref{eqn:new-d-4}), that 
\begin{align*}
\hwkb(x)&=\frac{\beta}{\alpha_k}\hwk(x)+(\wkb(x_k)-\wkak(x_k)) \\
&\leq \frac{\beta}{\alpha_k}\hwk(x)=\frac{\beta}{\alpha_0}(\hz(x)-\hz(0))+o(1)
\end{align*}
as $k\rightarrow\infty$, locally uniformly in $x\in{\bf R}^2$, for any $\beta\in I_+$, 
where $(r,\theta)$ denotes the polar coordinate in ${\bf R}^2$. 
Then, following the proof of Lemma \ref{lem:2-g}, we use the Pohozaev identity (\ref{eqn:2.24}), (\ref{eqn:new-d-15}), (\ref{eqn:new-d-14}) and the value $\bar{\lambda}$ given in (\ref{eqn:1.13}) to obtain 
\begin{equation}
\int_{I_+}\hpsi(\beta)\calP(d\beta)=\left(\int_{I_+}\hat{\phi}_0(\beta)\hpsi(\beta)\calP(d\beta)\right)^2, 
 \label{eqn:new-d-16}
\end{equation}
where 
\[
\hat{\phi}_0(\beta)=\frac{\sqrt{\alpha_0}}{\int_{I_+}\alpha\calP(d\alpha)}\beta. 
\]

{\it Step 4}.\ In this final step, we shall show that there exists a $\calP$-measurable set $J\subset I_+$ such that $\hpsi=\chi_J$ $\calP$-a.e. on $I_+$ and that 
\begin{equation}
\calP(J)>0,\quad \calP(J\cap\Iinf)=0,\quad \frac{\calP(J)}{\left(\int_J \beta\calP(d\beta)\right)^2}
=\frac{1}{\left(\int_{I_+} \beta\calP(d\beta)\right)^2}. 
 \label{eqn:new-d-17}
\end{equation}
The proof of the lemma is reduced to showing (\ref{eqn:new-d-17}) since (\ref{eqn:new-d-0}) and (\ref{eqn:new-d-17}) do not occur simultaneously by Lemma \ref{lem:new-e}. 

Noting that 
\[
\tpsi=\chi_{\Iinf}\quad \mbox{$\calP$-a.e. on $I_+$}, 
\] 
recall Lemma \ref{lem:new-b}, and that 
\[
B_{\tauk R}(x_k)\cap B_{\sigk R}=\emptyset,\quad \int_\Omega e^{\wkb}=1 
\]
for any $k\gg 1$, $\beta\in I_+\setminus\{0\}$ and $R>0$, we find 
\[
0\leq \tpsi+\hpsi\leq 1\quad \mbox{$\calP$-a.e. on $I_+$}, 
\]
and thus
\begin{equation}
\hpsi=0\quad \mbox{$\calP$-a.e. on $\Iinf$}.
 \label{eqn:new-d-18}
\end{equation}
We put 
\[
\hI=I_+\setminus\Iinf
\]
and see from (\ref{eqn:new-d-16}) and (\ref{eqn:new-d-18}) that 
\begin{equation}
\hd=\int_{\hI}\hpsi(\beta)\calP(d\beta)=\left(\int_{\hI}\hat{\phi}_0(\beta)\hpsi(\beta)\calP(d\beta)\right)^2>0. 
 \label{eqn:new-d-19}
\end{equation}
Let
\begin{align*}
&\hcalI(\psi)=\int_{\hI}\hat{\phi}_0(\beta)\psi(\beta)\calP(d\beta) \\
&\hcalC=\{\psi \ | \ 0\leq \psi(\beta)\leq 1 \ \mbox{$\calP$-a.e. on $\hI$ and } \int_{\hI}\psi(\beta)\calP(d\beta)=\hd \}. 
\end{align*}
Noting the monotonicity of $\hat{\phi}_0=\hat{\phi}_0(\beta)$ and complying the proof of Lemma \ref{lem:2-h}, we can show the following properties: 

(a) The value $\sup_{\psi\in\hcalC}\hcalI(\psi)$ is attained by 
\[
\psi_\ast(\beta)=\chi_{\{\hat{\phi}_0>\hs\}\cap\hI}(\beta)+\hc\chi_{\{\hat{\phi}_0=\hs\}}(\beta)
\]
with $\hs$ and $\hc$ defined by 
\begin{align*}
&\hs=\inf\{t \ | \ \calP(\{\hat{\phi}_0>t\}\cap\hI)\leq \hd\} \\
&\hc\calP(\{\hat{\phi}_0=\hs\})=\hd-\calP(\{\hat{\phi}_0>\hs\}\cap\hI),\quad 0\leq \hc \leq 1. 
\end{align*}

(b) The uniqueness holds in the sense that if $\psi_m\in\hcalC$ is the maximizer then $\psi_m=\psi_\ast$ $\calP$-a.e. on $\hI$. 

Following the argument to show (\ref{eqn:new-b-claim2}), which is developed in the first part of the present section, 
and using (\ref{eqn:optimalconst}), (\ref{eqn:1.13}), (\ref{eqn:new-d-18}) and properties (a)-(b), 
we find that there exists $\hb\in I_+$ such that 
\begin{equation}
\hpsi=\chi_{\hJ}\quad \mbox{$\calP$-a.e. on $\hI$}, 
 \label{eqn:new-d-20}
\end{equation}
where 
\[
\hJ=
\begin{cases}
[\hb,1]\setminus\Iinf & \mbox{if $\calP(\{\hb\})=0$ or if $\calP(\{\hb\})>0$ and $\hc=1$} \\
(\hb,1]\setminus\Iinf & \mbox{otherwise (i.e., $\calP(\{\hb\})>0$ and $\hc=0$)}. 
\end{cases}
\]
Note that either $\hc=0$ or $\hc=1$ holds if $\calP(\{\hb\})>0$. 

Consequently, (\ref{eqn:new-d-19}) and (\ref{eqn:new-d-20}) yield 
\[
\frac{\calP(\hat{J})}{\left(\int_{\hat{J}}\beta \calP(d\beta)\right)^2}=\frac{\alpha_0}{\left(\int_{I_+}\alpha \calP(d\alpha)\right)^2}\leq\frac{1}{\left(\int_{I_+}\alpha \calP(d\alpha)\right)^2}, 
\]
which implies 
\[
\frac{\calP(\hat{J})}{\left(\int_{\hat{J}}\beta \calP(d\beta)\right)^2}=\frac{1}{\left(\int_{I_+}\alpha \calP(d\alpha)\right)^2}
\]
by $\alpha_0\leq 1$ and (\ref{eqn:1.13}). 
Hence (\ref{eqn:new-d-17}) is shown for $J=\hat{J}$. 
\end{proof}

\begin{lem}
There exist $t\in(0,1)$ and $C_{\const\insb}>0$ such that 
\[
\sup_{\partial B_r}\wka\leq C_{\subb}+t\inf_{\partial B_r}\wka-2(1-t)\log r
\]
for any $r\in[2r',R_0]$, $r'\in(0,R_0/2]$, $\alpha\in I_+$ and $k\gg 1$, 
where $t$ and $C_{\subb}$ are independent of $r$, $r'$, $R_0$, $\alpha$ and $k\gg 1$. 
 \label{lem:new-f}
\end{lem}

\begin{proof} We comply \cite{ls94}. 
Fix $r\in[2r',R_0]$ and $r'\in(0,R_0/2]$, and put
\[
\zka(x)=\wka(rx)+2\log r
\]
for $\alpha\in I_+$ and $k$. 
Then it holds that 
\begin{equation}
-\Delta\zka=\alpha\lambda_ke^{\xik(rx)}\int_{I_+}\beta\left(e^{\zkb}-\frac{1}{|\Omega|}\right)\calP(d\beta)\quad \mbox{in $B_2\setminus\overline{B_{1/2}}$}. 
 \label{eqn:new-f-1}
\end{equation}
It follows from Lemma \ref{lem:new-d} that 
\begin{equation}
\zka(x)=(\wka(rx)+2\log(r|x|))-2\log|x| \leq C_{\suba}+2\log 2
 \label{eqn:new-f-2}
\end{equation}
for any $x\in B_2\setminus\overline{B_{1/2}}$, $\alpha\in I_+$ and $k\gg 1$. 
Thus there exists $C_{\const\insj}>0$, independent of $r$, $r'$, $R_0$, $\alpha$ and $k\gg 1$, such that
\begin{align}
|\mbox{[$r.h.s.$ of (\ref{eqn:new-f-1})]}|
&\leq \lambda_k\sup_{B_{2R_0}}e^{\xik}\left(\frac{1}{|\Omega|}+\sup_{\beta\in I_+}\sup_{x\in B_2\setminus\overline{B_{1/2}}}e^{\zkb}\right) \nonumber\\
&\leq \lambda_k\sup_{B_{2R_0}}e^{\xik}\left(\frac{1}{|\Omega|}+4e^{C_{\suba}}\right)\leq C_{\subj}\quad \mbox{in $B_2\setminus\overline{B_{1/2}}$}
 \label{eqn:new-f-3}
\end{align}
for $\alpha\in I_+$ and $k\gg 1$ by (\ref{eqn:new-f-2}). 

Let $\zka'=\zka'(x)$ be the unique solution to
\begin{align*}
&-\Delta\zka'=\alpha\lambda_k e^{\xik(rx)}\int_{I_+}\beta\left(e^{\zkb}-\frac{1}{|\Omega|}\right)\calP(d\beta)\quad \mbox{in $B_2\setminus\overline{B_{1/2}}$}, \\ 
&\zka'=0\quad \mbox{on $\partial(B_2\setminus\overline{B_{1/2}})$}. 
\end{align*}
The elliptic regularity and (\ref{eqn:new-f-3}) admit $C_{\const\insk}>0$, independent of $r$, $r'$, $R_0$, $\alpha$ and $k\gg 1$, such that
\begin{equation}
|\zka'|\leq C_{\subk}\quad \mbox{in $B_2\setminus\overline{B_{1/2}}$}
 \label{eqn:new-f-5}
\end{equation}
for $\alpha\in I_+$ and $k\gg 1$. 
Here we introduce 
\[
\hka(x)=C_{\const\insl}+(\zka'(x)-\zka(x)),\quad C_{\subl}=C_{\suba}+2\log 2+C_{\subk}
\]
in view of (\ref{eqn:new-f-2}) and (\ref{eqn:new-f-5}). 
The maximum principle assures that $\hka=\hka(x)$ is the non-negative harmonic function on $B_2\setminus\overline{B_{1/2}}$, 
and then the Harnack inequality admits a universal constant $t\in (0,1)$ such that 
\[
t\sup_{\partial B_1}\hka\leq\inf_{\partial B_1}\hka 
\]
or 
\begin{equation}
t\sup_{\partial B_1}(\zka'-\zka)\leq (1-t)C_{\subl}+\inf_{\partial B_1}(\zka'-\zka)
 \label{eqn:new-f-6}
\end{equation}
for $\alpha\in I_+$ and $k\gg 1$. 
Combining (\ref{eqn:new-f-5}) and (\ref{eqn:new-f-6}) shows 
\[
-tC_{\subk}-t\inf_{\partial B_1}\zka\leq (1-t)C_{\subl}+C_{\subk}-\sup_{\partial B_1}\zka, 
\]
which means the lemma for $C_{\subb}=(1+t)C_{\subk}+(1-t)C_{\subl}$. 
\end{proof}

\begin{lem}
There exist $\varepsilon_\ast>0$, $R_\ast>0$ and $C_i>0$ ($i=\const\insc,\const\insd$) such that 
\begin{equation}
\wka(0)+C_{\subc}\inf_{\partial B_r}\wka+2(1+C_{\subc})\log r\leq C_{\subd}
 \label{eqn:new-g}
\end{equation}
for any $r\in(0,R_\ast]$, $\alpha\in[\binf-\varepsilon_\ast,1]$ and $k\gg 1$, 
where $\varepsilon_\ast$, $R_\ast$, $C_{\subc}$ and $C_{\subd}$ are independent of $r$, $\alpha$ and $k\gg 1$. 
 \label{lem:new-g}
\end{lem}

\begin{proof} At first, we note that there exists $\delta=\delta(\calP,\Iinf)>0$ such that 
\begin{equation*}
\binf=(1+\delta)\frac{\int_{\Iinf}\beta\calP(d\beta)}{2\calP(\Iinf)}
\end{equation*}
since $\binf>\int_{\Iinf}\beta\calP(d\beta)/(2\calP(\Iinf))$ by Lemma \ref{lem:2-d}, (\ref{eqn:2-1}), Lemma \ref{lem:new-b} and (\ref{eqn:new-d-0}). 
We put 
\[
D=\frac{2}{\delta}
\]
and introduce the auxiliary function 
\[
\Pka(r)=\wka(0)+\frac{D}{2\pi r}\int_{\partial B_r}\wka ds+2(1+D)\log r 
\]
inspired by \cite{bls93, shafrir92}. 
Since
\[
\frac{d}{dr}\left(\frac{1}{2\pi r}\int_{\partial B_r}\wka ds\right)=\frac{1}{2\pi r}\int_{\partial B_r}\frac{\partial\wka}{\partial \nu} ds, 
\]
it holds that 
\begin{equation}
\frac{d\Pka}{dr}(r)\leq \frac{D\lambda_k}{2\pi r}\Qka(r), 
 \label{eqn:new-g-2}
\end{equation}
for $r\in(0,R_0]$ and $\alpha\in I_+$, where $\nu$ is the outer unit normal vector and 
\begin{align*}
\Qka(r)&=\frac{4\pi(1+D)}{D\lambda_k}+\frac{1}{|\Omega|}\int_{B_r}e^{\xik}dx\cdot\int_{I_+}\beta\calP(d\beta) \\
&\quad -\alpha\int_{\Iinf}\beta\left(\int_{B_r}e^{\wkb+\xik}dx\right)\calP(d\beta). 
\end{align*}

Given $\varepsilon>0$ whose range is determined later on, 
there exists $R_\varepsilon=R_\varepsilon(\calP,\Omega)>0$ such that
\begin{equation}
\frac{1}{|\Omega|}\int_{B_{R_\varepsilon}}e^{\xik}dx\cdot\int_{I_+}\beta\calP(d\beta)\leq\varepsilon
 \label{eqn:new-g-3}
\end{equation}
for any $k$. 
We may assume that $R_\varepsilon$ is monotone increasing in $\varepsilon$. 
We also have $L_\varepsilon>0$, independent of $r$ and $k$, such that 
\begin{equation}
\int_{\Iinf}\beta\left(\int_{B_r}e^{\wkb+\xik}dx\right)\calP(d\beta)\geq \int_{\Iinf}\beta\calP(d\beta)-\varepsilon
 \label{eqn:new-g-4}
\end{equation}
for any $r\geq \sigk L_\varepsilon$ and $k\gg 1$ by the definition of $\tpsi$, Lemma \ref{lem:new-b} and the convergence (\ref{eqn:convergence-twkb}). 
We may assume that $L_\varepsilon$ is monotone decreasing in $\varepsilon$. 
It is clear that 
\begin{equation}
\frac{4\pi(1+D)}{D\lambda_k}\leq \frac{4\pi(1+D)}{D\bar{\lambda}}+\varepsilon
 \label{eqn:new-g-5}
\end{equation}
for $k\gg 1$. 
Properties (\ref{eqn:new-g-3})-(\ref{eqn:new-g-5}) imply 
\begin{equation}
\Qka(r)\leq 2\varepsilon+\frac{4\pi(1+D)}{D\bar{\lambda}}-(\binf-\varepsilon)\left(\int_{\Iinf}\beta\calP(d\beta)-\varepsilon\right)
 \label{eqn:new-g-6}
\end{equation}
for any $r\in[\sigk L_\varepsilon,R_\varepsilon]$, $\alpha\in[\binf-\varepsilon,1]$ and $k\gg 1$. 

We now examine the range of $\varepsilon$ such that the right-hand-side of (\ref{eqn:new-g-6}) is non-positive. 
It follows from (\ref{eqn:1.13}) and (\ref{eqn:new-d-0}) that  
\begin{equation}
\frac{4\pi(1+D)}{D\bar{\lambda}}=(1+1/D)\frac{\left(\int_{\Iinf}\beta\calP(d\beta)\right)^2}{2\calP(\Iinf)}. 
 \label{eqn:new-g-7}
\end{equation}
We use (\ref{eqn:new-d-0}), (\ref{eqn:new-g-7}) and $D=2/\delta$ to obtain 
\begin{align*}
&\mbox{[r.h.s. of (\ref{eqn:new-g-6})]} \\
&=-\varepsilon^2+\left\{2+\left(1+\frac{1+\delta}{2\calP(\Iinf)}\right)\int_{\Iinf}\beta\calP(d\beta)\right\}\varepsilon
-\frac{\delta\left(\int_{\Iinf}\beta\calP(d\beta)\right)^2}{4\calP(\Iinf)}, 
\end{align*}
and therefore, there exists $\varepsilon_\ast=\varepsilon_\ast(\calP,\Iinf)>0$ such that 
\begin{equation}
\mbox{[r.h.s. of (\ref{eqn:new-g-6})]}\leq 0
 \label{eqn:new-g-8}
\end{equation}
for any $0<\varepsilon<\varepsilon_\ast$. 

Noting that $\Qka(r)$ is independent of $\varepsilon$, we organize (\ref{eqn:new-g-2}), (\ref{eqn:new-g-6}) and (\ref{eqn:new-g-8}), 
so that $\Pka'(r)\leq 0$ for any $r\in[\sigk L_{\varepsilon_\ast},R_{\varepsilon_\ast}]$, $\alpha\in[\binf-\varepsilon_\ast,1]$ and $k\gg 1$. 
This implies 
\begin{equation}
\sup_{0<r\leq R_\ast}\Pka(r)=\sup_{0<r\leq \sigk L_\ast}\Pka(r)
 \label{eqn:new-g-9}
\end{equation}
for $\alpha\in[\binf-\varepsilon_\ast,1]$ and $k\gg 1$, where $R_\ast=R_{\varepsilon_\ast}$ and $L_\ast=L_{\varepsilon_\ast}$. 
Using $\wka\leq \wka(0)\leq \wk(0)$ valid for any $\alpha\in I_+$, we estimate $\Pka$ by 
\begin{align}
\Pka(r)&=(1+D)\wka(0)+\frac{D}{2\pi r}\int_{\partial B_r}(\wka-\wka(0))ds+2(1+D)\log r \nonumber\\
&\leq (1+D)(\wka(0)+2\log\sigk L_\ast)\leq 2(1+D)\log L_\ast
 \label{eqn:new-g-10}
\end{align}
for $r\in(0,\sigk L_\ast]$, $\alpha\in[\binf-\varepsilon_\ast,1]$ and $k\gg 1$. 

Finally, we obtain $C_{\subc}=D=2/\delta$ and $C_{\subd}=2(1+D)\log L_\ast=2(1+2/\delta)\log L_\ast$ 
by (\ref{eqn:new-g-9}), (\ref{eqn:new-g-10}) and $\mbox{[l.h.s. of (\ref{eqn:new-g})]}\leq \Pka(r)$, 
provided that $\varepsilon_\ast$ and $R_\ast$ are given above. 
\end{proof} 

We are now in a position to prove Proposition \ref{pro:7}. \\ 

\underline{{\it Proof of Proposition \ref{pro:7}:}} \ Fix $\alpha_0\in I_+$ such that 
\[
\begin{cases}
\alpha_0\in[\max\{\amin,\binf-\varepsilon_\ast\},\binf) & \mbox{if $\binf>\amin$} \\
\alpha_0=\amin>\binf-\varepsilon_\ast & \mbox{if $\binf=\amin$}, 
\end{cases}
\]
recall that $\binf$, $\varepsilon_\ast$ and $\amin$ are as in (\ref{eqn:2.21h}), Lemma \ref{lem:new-g} and (\ref{eqn:amin}), respectively. 
Note that $\calP(I_+\setminus\Iinf)>0$ and that $\Iinf=(\binf,1]$ and $\calP(\{\alpha_{\min}\})>0$ if $\binf=\amin$. 
It follows from Lemma \ref{lem:new-c} and the uniform boundedness of $\xik$ that 
\begin{equation}
\lim_{k\rightarrow\infty}\int_{B_{\sigkao R}}e^{\wkao+\xik}=0
 \label{eqn:finish-1}
\end{equation}
for any $R>0$, where $\sigma_{k,\alpha_0}=e^{-w_{k,\alpha_0}(0)/2}$. 
In addition, the residual vanishing, the uniform boundedness of $\xik$ and the monotonicity (\ref{eqn:monotone-1}) imply  
\begin{equation}
\lim_{k\rightarrow\infty}\int_{B_{2R_0}\setminus B_{R_\ast}}e^{\wkao+\xik}=0,\quad 
\lim_{k\rightarrow\infty}\int_{\Omega\setminus\Psi_k^{-1}(B_{2R_0})}e^{\wkao}=0, 
 \label{eqn:finish-2}
\end{equation}
where $R_\ast$ is as in Lemma \ref{lem:new-g}. 

Next, we shall prove 
\begin{equation}
\lim_{k\rightarrow\infty}\int_{B_{R_\ast}\setminus B_{\sigkao}}e^{\wkao+\xik}=0. 
 \label{eqn:finish-3}
\end{equation}
For any $r=|x|\in[\sigkao,R_\ast]$, we calculate 
\begin{align*}
\wkao(x)&\leq\sup_{\partial B_r}\wkao\leq C_{\subb}+t\inf_{\partial B_r}\wkao-2(1-t)\log r \\
&\leq C_{\subb}+\frac{t}{C_{\subc}}\{-\wkao(0)-2(1+C_{\subc})\log r+C_{\subd}\}-2(1-t)\log r \\
&=-s\wkao(0)-2(1+s)\log r+C_{\const\insm}, 
\end{align*}
using Lemmas \ref{lem:new-f}-\ref{lem:new-g}, where 
\[
s=t/C_{\subc},\quad C_{\subm}=C_{\subb}+tC_{\subd}/C_{\subc}. 
\]
Hence it holds that 
\begin{align*}
\int_{B_{R_\ast}\setminus B_{\sigkao}}e^{\wkao+\xik}
&\leq \sup_{B_{R_0}}e^{\xik}\cdot e^{C_{\subm}-s\wkao(0)}\int_{B_{R_\ast}\setminus B_{\sigkao}}|x|^{-2(1+s)}dx \\
&\leq C_{\const\insn}e^{-2s\wkao(0)}\int_1^\infty r^{-(1+2s)} dr \rightarrow 0 
\end{align*}
as $k\rightarrow\infty$, where 
\[
C_{\subn}=2\pi e^{C_{\subm}}\sup_k\sup_{B_{R_0}}e^{\xik}. 
\]

Consequently, (\ref{eqn:finish-1})-(\ref{eqn:finish-3}) yield
\[
\lim_{k\rightarrow\infty}\int_\Omega e^{\wkao}=0, 
\]
which is impossible since $\int_\Omega e^{\wkao}=1$ for any $k$. 
The proof is complete. \qed\\ 

We conclude this section with the following proposition.  

\begin{pro}
Under the assumptions of Theorem \ref{thm:main} it holds that 
\begin{equation} 
\amin>\frac{1}{2}\int_{I_+}\beta \calP(d\beta)=\frac{2}{\tgam}. 
 \label{eqn:1.12h}
\end{equation} 
 \label{pro:6}
\end{pro} 

\begin{proof} 
It suffices to show that $\binf=\amin$. 
Indeed, if this is the case, (\ref{eqn:1.12h}) follows from Lemma \ref{lem:2-d} and (\ref{eqn:2.17}). 
Since $\binf\geq \amin$ is obvious, we assume the contrary, $\binf>\amin$. 
Then it holds that $\supp \ \tpsi \subset [\binf, 1]$ by the definitions of $\binf$ and $\tpsi$, 
and thus we obtain $\calP([\amin,(\binf+\amin)/2])>0$ and $\tpsi=0$ $\calP$-a.e. on $[\amin,(\binf+\amin)/2]$. 
However, this is impossible by (\ref{eqn:2.17'}). 
\end{proof}

\section{Proof of Proposition \ref{pro:yyl-type-estimate}}\label{sec:yyl}

Henceforth, we put 
\[
\bar{w}=\frac{1}{|\Omega|}\int_\Omega w.
\]
Let $G=G(x,y)$ be the Green function: 
\[ -\Delta_x G(\cdot,y)=\delta_y-\frac{1}{|\Omega|} \quad \mbox{in $\Omega$}, \quad \int_\Omega G(x,y) dx=0. \]
We begin with the following lemma. 

\begin{lem}
It holds that 
\begin{equation} 
\wka-\bar{w}_{k,\alpha} \rightarrow \alpha\bar{\lambda}\left(\int_{I_+}\beta\calP(d\beta)\right)G(\cdot,x_0) \quad \mbox{in $C^2_{loc}(\Omega\setminus\{x_0\})$}.
 \label{eqn:3.1}
\end{equation} 
For every $\omega\subset\subset\Omega\setminus\{x_0\}$, there exists $C_{\const\insa,\omega}>0$, independent of $k$ and $\alpha$, such that 
\begin{equation} 
\underset{\omega}{\rm osc}\ \wka\equiv \sup_\omega \wka-\inf_\omega \wka \leq C_{\suba,\omega}. 
 \label{eqn:3.2}
\end{equation} 
 \label{lem:3-1}
\end{lem}

\begin{proof} Since 
\[ \wka(x)-\bar{w}_{k,\alpha}=\alpha\int_\Omega G(x,y)\left\{\lambda_k\int_{I_+}\beta e^{\wkb(y)}\calP(d\beta)\right\}dy \]
and 
\[ \lambda_k\int_{I_+}\beta e^{\wkb}\calP(d\beta) \overset{\ast}{\rightharpoonup} \bar{\lambda}\left(\int_{I_+}\beta\calP(d\beta)\right)\delta_{x_0} \]
by (\ref{eqn:concentration}) with $s=0$ and ${\cal S}=\{x_0\}$, recalling Proposition \ref{pro:4}, we have  
\begin{equation*}
\wka-\bar{w}_{k,\alpha}\rightarrow \alpha\bar{\lambda}\left(\int_{I_+}\beta\calP(d\beta)\right)G(\cdot,x_0) 
\end{equation*}
locally uniformly in $\Omega\setminus\{x_0\}$. Then the standard argument of elliptic regularity implies (\ref{eqn:3.1}) and (\ref{eqn:3.2}).  
\end{proof} 

We decompose $\wk$ as $\wk=\wk^{(1)}+\wk^{(2)}+\wk^{(3)}$, using the solutions $\wk^{(1)}$, $\wk^{(2)}$ and $\wk^{(3)}$ to 
\begin{align*}
&-\Delta \wk^{(1)}=g_k \quad \mbox{in $B_{2R_0}$}, \quad  \wk^{(1)}=0 \quad \mbox{on $\partial B_{2R_0}$} \\
&-\Delta \wk^{(2)}=h_k \quad \mbox{in $B_{2R_0}$}, \quad  \wk^{(2)}=0 \quad \mbox{on $\partial B_{2R_0}$} \\
&-\Delta \wk^{(3)}=0 \quad \mbox{in $B_{2R_0}$}, \quad  \wk^{(3)}=\wk \quad \mbox{on $\partial B_{2R_0}$}, 
\end{align*}
where 
\begin{align*}
&g_k=g_k(x)\equiv \lambda_k\int_{I_+}\beta e^{\wkb(x)+\xik(x)} \calP(d\beta) \\
&h_k=h_k(x)\equiv -\frac{\lambda_k}{\vert \Omega\vert}\int_{I_+}\beta\calP(d\beta)e^{\xik(x)}. 
\end{align*}
By the elliptic regularity there exists $C_{\const\insc}>0$ independent of $k$ such that 
\[ -C_{\subc}\leq \wk^{(2)}\leq 0 \quad \mbox{in $B_{2R_0}$}. \]
By the maximum principle and Lemma \ref{lem:3-1}, we also have $C_{\const\insd}>0$ independent of $k$ such that 
\[ \underset{\bar{B}_{2R_0}}{\rm{osc}} \wk^{(3)}\leq C_{\subd}. \]
Thus it holds that 
\begin{equation}
\wk(x)-\wk(0)=\wk^{(1)}(x)-\wk^{(1)}(0)+O(1)
 \label{eqn:3-1}
\end{equation}
as $k\rightarrow\infty$ uniformly in $x\in B_{2R_0}$. 

Let $G_0=G_0(x,y)$ be the another Green function defined by 
\[ -\Delta_x G_0(\cdot,y)=\delta_y \quad \mbox{in $B_{2R_0}$}, \quad G_0(\cdot,y)=0 \quad \mbox{on $\partial B_{2R_0}$}. \]
Then it holds that 
\begin{equation}
\wk^{(1)}(x)-\wk^{(1)}(0)=\int_{B_{2R_0}} (G_0(x,y)-G_0(0,y))g_k(y)dy
 \label{eqn:3-1'}
\end{equation}
for $x\in B_{2R_0}$.  We have, more precisely, 
\[ G_0(x,y)=\left\{ \begin{array}{ll} 
\Gamma(\vert x-y\vert)-\Gamma(\frac{\vert y\vert}{2R_0}\vert x-\bar{y}\vert), & y\neq 0, \ y\neq x \\
\Gamma(\vert x\vert )-\Gamma(2R_0) & y=0, \ y\neq x,
\end{array} \right.  \] 
using the fundamental solution and the Kelvin transformation: 
\[ \Gamma(\vert x\vert)=\frac{1}{2\pi}\log\frac{1}{\vert x\vert},\quad \bar{y}=\left(\frac{2R_0}{\vert y\vert}\right)^2 y, \] 
which implies 
\[ G_0(x,y)-G_0(0,y)=\frac{1}{2\pi}\log\frac{\vert y\vert}{\vert x-y\vert}-\frac{1}{2\pi}\log\frac{\vert \bar{y}\vert}{\vert x-\bar{y}\vert} \] 
for $y\in B_{2R_0}$ satisfying $y\neq x$ and $y\neq 0$. 

Since
\[ \frac{2}{3}\leq\frac{\vert \bar{y}\vert}{\vert x-\bar{y}\vert}\leq 2, \quad x\in B_{R_0}, \ y\in B_{2R_0}\setminus\{0\}, \] 
and since 
\[ 0\leq \int_{B_{2R_0}} g_k \leq \lambda_k\int_{I_+}\beta\calP(d\beta)\cdot \sup_{B_{R_0}}e^{\xik}=O(1), \]
we end up with 
\begin{align}
&\int_{B_{2R_0}} (G_0(x,y)-G_0(0,y))g_k(y)dy \nonumber\\ 
&\quad =\frac{1}{2\pi}\int_{B_{2R_0}}g_k(y)\log\frac{\vert y\vert}{\vert x-y\vert}dy+O(1) 
 \label{eqn:3-1''}
\end{align}
as $k\rightarrow\infty$ uniformly in $x\in B_{R_0}$. 

Consequently, (\ref{eqn:3-1})-(\ref{eqn:3-1''}) yield 
\begin{equation*}
\wk(x)-\wk(0)=\frac{1}{2\pi}\int_{B_{2R_0}}g_k(y)\log\frac{\vert y\vert}{\vert x-y\vert}dy+O(1)
\end{equation*}
as $k\rightarrow\infty$ uniformly in $x\in B_{R_0}$. 
This means 
\begin{align}
&\twk(x)=\frac{1}{2\pi}\int_{B_{2R_0/\sigma_k}}\sigma_k^2g_k(y)\log\frac{\vert y\vert}{\vert \sigk x-y\vert}dy+O(1) \nonumber\\ 
&=\frac{1}{2\pi}\int_{B_{2R_0/\sigma_k}}\tilde{f}_k(y)\log\frac{\vert y\vert}{\vert x-y\vert}dy+O(1)
 \label{eqn:3-2'}
\end{align}
as $k\rightarrow\infty$ uniformly in $x\in B_{R_0/\sigma_k}$, where $\tilde{f}_k=\tilde{f}_k(y)$ is as in (\ref{eqn:def-tilde-f_n}). 

Let $\tgam$ be as in (\ref{eqn:2.17}), and put 
\begin{equation}
\tgamk=\frac{1}{2\pi}\int_{B_{2R_0/\sigma_k}}\tilde{f}_k. 
 \label{eqn:def-tilgam}
\end{equation}
To employ the argument of \cite{csl07}, we prepare the following lemma with which $\tgamk$ and $\tgam$ are connected. 
\begin{lem}
It holds that 
\begin{equation} 
\lim_{k\rightarrow\infty}\tgamk=\tgam. 
 \label{eqn:3.9}
\end{equation} 
\end{lem}

\begin{proof} From (\ref{eqn:def-tilde-f_n}), $\int_\Omega e^{\wkb}=1$, $\lambda_k\uparrow\bar{\lambda}$ and (\ref{eqn:1.13}), it follows that 
\begin{equation}
\tgamk=\frac{1}{2\pi}\int_{B_{2R_0/\sigma_k}}\tilde{f}_k\leq \frac{\lambda_k}{2\pi}\int_{I_+}\beta\calP(d\beta)\leq \tgam
 \label{eqn:gam-estimate-above}
\end{equation}
for any $k$. 
On the other hand, given $\varepsilon>0$, we have $L_\varepsilon>0$ such that 
\[
\liminf_{k\rightarrow\infty}\tgamk
\geq \liminf_{k\rightarrow\infty}\left(\frac{1}{2\pi}\int_{B_{L_\varepsilon}}\tilde{f}_k\right)
\geq \tgam-\varepsilon
\]
by (\ref{eqn:entire-conv}), (\ref{eqn:entire-mass}) and (\ref{eqn:2.17}). 
\end{proof}

\begin{lem}
For every $0<\varepsilon\ll 1$, there exist $\tilde{R}_\varepsilon\geq 2$ and $C_{\const\insc,\varepsilon}>0$ such that 
\begin{equation} 
\twk(x)\leq -(\tgamk-\varepsilon)\log \vert x\vert+C_{\subc,\varepsilon} 
 \label{eqn:3.11}
\end{equation} 
for $k\gg 1$ and $x\in B_{R_0/\sigma_k}\setminus B_{\tilde{R}_\varepsilon}$. 
 \label{lem:3-2}
\end{lem}

\begin{proof}  By (\ref{eqn:3.9}) and (\ref{eqn:entire-conv}), given $0<\varepsilon\ll 1$, we can take $\tilde{R}_\varepsilon\geq 2$ such that 
\begin{equation}
\frac{1}{2\pi}\int_{B_{\tilde{R}_\varepsilon/2}}\tilde{f}_k\geq \tgamk-\varepsilon/3
 \label{eqn:3-3}
\end{equation}
for $k\gg 1$. 
It follows from (\ref{eqn:3-2'}) that 
\begin{equation}
\twk(x)=K^1_k(x)+K^2_k(x)+K^3_k(x)+O(1), \quad k\rightarrow\infty 
 \label{eqn:3-4}
\end{equation}
uniformly in $x\in B_{R_0/\sigma_k}\setminus B_{\tilde{R}_\varepsilon}$, where
\begin{align*}
&K^1_k(x)=\frac{1}{2\pi}\int_{B_{\tilde{R}_\varepsilon/2}}\tilde{f}_k(y)\log\frac{\vert y\vert}{\vert x-y\vert}dy \\ 
&K^2_k(x)=\frac{1}{2\pi}\int_{B_{\vert x\vert/2}(x)}\tilde{f}_k(y)\log\frac{\vert y\vert}{\vert x-y\vert}dy \\
&K^3_k(x)=\frac{1}{2\pi}\int_{B'(x)}\tilde{f}_k(y)\log\frac{\vert y\vert}{\vert x-y\vert}dy 
\end{align*}
for $B'(x)=B_{2R_0/\sigma_k}\setminus(B_{\tilde{R}_\varepsilon/2}\cup B_{\vert x\vert/2}(x))$. 

Since 
\[ \frac{\vert y\vert}{\vert x-y\vert}\leq 2\frac{\vert y\vert}{\vert x\vert}\leq \frac{\tilde{R}_\varepsilon}{\vert x\vert}, \quad y\in B_{\tilde{R}_\varepsilon/2}, \ x\in B_{R_0/\sigma_k}\setminus B_{\tilde{R}_\varepsilon}, \] 
there exists $C_{\const\insd,\varepsilon}>0$ independent of $k\gg 1$ and $x$ such that  
\begin{equation}
K^1_k(x)\leq \frac{1}{2\pi}(\log\tilde{R}_\varepsilon-\log\vert x\vert)\int_{B_{\tilde{R}_\varepsilon/2}}\tilde{f}_k \leq C_{\subd,\varepsilon}-(\tgamk-\varepsilon/3)\log\vert x\vert
 \label{eqn:3-5}
\end{equation}
for $k\gg 1$ and $x\in B_{R_0/\sigma_k}\setminus B_{\tilde{R}_\varepsilon}$ by (\ref{eqn:3-3}). 
We also have 
\[ \frac{\vert y\vert}{\vert x-y\vert}\leq 3, \quad y\in B_{2R_0/\sigma_k}\setminus B_{\vert x\vert/2}(x), \] 
and hence 
\begin{align}
K^3_k(x)&\leq \frac{\log 3}{2\pi}\int_{B'(x)}\tilde{f}_k \leq \frac{\log 3}{2\pi}\Vert \tilde{f}_k\Vert_{L^1(B_{2R_0/\sigma_k})} \nonumber\\
&\leq \frac{\lambda_k\log 3}{2\pi}\int_{I_+}\beta\calP(d\beta)\cdot \sup_{B_{R_0}}e^{\xik}
 \label{eqn:3-5'}
\end{align}
for $k\gg 1$ and $x\in B_{R_0/\sigma_k}\setminus B_{\tilde{R}_\varepsilon}$. 

Now we take 
\[ D_1(x)=B_{1/\vert x\vert}(x),\quad D_2(x)=B_{\vert x\vert/2}(x)\setminus B_{1/\vert x\vert}(x). \]
Since 
\[ \vert y\vert<\vert x\vert+1/\vert x\vert, \quad y\in D_1(x) \] 
and 
\[ \frac{\vert y\vert}{\vert x-y\vert}\leq \frac{3}{2}|x|^2, \quad y\in D_2(x), \ x\in{\bf R}^2\setminus B_{\sqrt{2}}, \] 
we have 
\begin{align}
&K^2_k(x)=\frac{1}{2\pi}\int_{D_1(x)\cup D_2(x)}\tilde{f}_k(y)\log\frac{\vert y\vert}{\vert x-y\vert}dy \nonumber\\
&\quad \leq \frac{1}{2\pi}\int_{D_1(x)}\tilde{f}_k(y)\log\frac{\vert x\vert+1/\vert x\vert}{\vert x-y\vert}dy +\frac{2\log|x|+\log (3/2)}{2\pi}\int_{D_2(x)}\tilde{f}_k \nonumber\\
&\quad \leq \frac{\Vert \tilde{f}_k\Vert_{L^\infty(D_1(x))}}{2\pi}\int_{D_1(x)}\log\frac{1}{\vert x-y\vert}dy
+\frac{\log(3\vert x\vert/2)}{2\pi}\int_{D_1(x)}\tilde{f}_k \nonumber\\
&\quad +\frac{2\log\vert x\vert+\log (3/2)}{2\pi}\int_{D_2(x)}\tilde{f}_k \nonumber\\
&\quad \leq \frac{\Vert\tilde{f}_k\Vert_{L^\infty(D_1(x))}}{2\pi}\int_{D_1(x)}\log\frac{1}{\vert x-y\vert}dy+\frac{\log (3/2)}{2\pi}\int_{B_{|x|/2}(x)}\tilde{f}_k \nonumber\\ 
&\quad +\frac{\log\vert x\vert}{\pi}\int_{B_{\vert x\vert/2}(x)}\tilde{f}_k \leq C_{\const\inse}+\frac{2\varepsilon}{3}\log\vert x\vert 
 \label{eqn:3-5''}
\end{align}
for some $C_{\sube}>0$ independent of $x\in B_{R_0/\sigma_k}\setminus B_{\tilde{R}_\varepsilon}$, $k\gg 1$, and $\varepsilon$. 

Here, the last inequality of (\ref{eqn:3-5''}) follows from (\ref{eqn:3-3}) and (\ref{eqn:3.9}). 
Properties (\ref{eqn:3-4})-(\ref{eqn:3-5''}) imply (\ref{eqn:3.11}). 
\end{proof}

\begin{lem}
It holds that 
\begin{equation}
\int_{B_{2R_0/\sigma_k}}\tilde{f}_k(y)\left\vert \log\vert y\vert \right\vert dy=O(1) \quad \mbox{as $k\rightarrow\infty$}. 
 \label{eqn:3-6}
\end{equation}
 \label{lem:3.4}
\end{lem} 

\begin{proof} \ By (\ref{eqn:1.12h}) and (\ref{eqn:3.9}), there exist $\varepsilon_0>0$ and $\delta_0>0$ such that 
\begin{equation}
-\amin(\tgamk-\varepsilon_0/2)\leq -(2+3\delta_0) 
 \label{eqn:3-6'}
\end{equation}
for $k\gg 1$. Let 
\begin{equation}
\tilde{R}_0=\tilde{R}_{\varepsilon_0/2} 
 \label{eqn:def-tilde-R0}
\end{equation}
for $\tilde{R}_\varepsilon$ as in Lemma \ref{lem:3-2} with $\varepsilon=\varepsilon_0/2$. 
Then, by (\ref{eqn:relation-w})-(\ref{eqn:2.12}), (\ref{eqn:3.11}) and (\ref{eqn:3-6'}) we obtain $C_{\const\insf,\varepsilon_0}>0$ such that 
\begin{align}
&\tilde{f}_k(y)=\lambda_k\int_{I_+}\beta e^{\twkb(y)+\txik(y)}\calP(d\beta) \nonumber\\ 
&\quad \leq \lambda_k\int_{I_+}\beta e^{\beta\twk(y)+\txik(y)}\calP(d\beta) \nonumber\\
&\quad \leq \lambda_k\int_{I_+}\exp\left[-\beta\{(\tgamk-\varepsilon_0/2)\log\vert y\vert-C_{\subc,\varepsilon_0}\}+\sup_{B_{2R_0}}\xik\right]\calP(d\beta) \nonumber\\
&\quad \leq C_{\subf,\varepsilon_0}\vert y\vert^{-(2+3\delta_0)} 
 \label{eqn:3-6''}
\end{align}
for $k\gg 1$ and $y\in B_{R_0/\sigma_k}\setminus B_{\tilde{R}_0}$. 

Therefore, we obtain $C_{\const\insg,\varepsilon_0,\delta_0}>0$ independent of $k\gg 1$ such that 
\begin{align*}
&\int_{B_{2R_0/\sigma_k}}\tilde{f}_k(y)\left\vert \log\vert y\vert\right\vert dy \leq \Vert \tilde{f}_k\Vert_{L^\infty(B_{2R_0/\sigma_k})}\int_{B_{\tilde{R}_0}}\left\vert \log\vert y\vert\right\vert dy \\ 
&\quad +C_{\subf,\varepsilon_0}\int_{{\bf R}^2\setminus B_{\tilde{R}_0}}\vert y\vert^{-(2+3\delta_0)}\log\vert y\vert dy \leq C_{\subg,\varepsilon_0,\delta_0} 
\end{align*}
for $k\gg 1$, which means (\ref{eqn:3-6}). 
\end{proof} 

\begin{lem}
There exists $\delta_0>0$ such that
\[ \twk(x)=-\tgamk\log\vert x\vert+O(1) \quad \mbox{as $k\rightarrow\infty$} \]
uniformly in $x\in B_{R_0/\sigma_k}\setminus B_{(\log\sigma_k^{-1})^{1/\delta_0}}$. 
 \label{lem:3-3}
\end{lem}

\begin{proof} Let $\varepsilon_0>0$ and $\delta_0>0$ satisfy (\ref{eqn:3-6'}) and consider
\begin{equation} 
\tgamk'(x)=\frac{1}{2\pi}\int_{B_{\vert x\vert/2}}\tilde{f}_k 
 \label{eqn:3.21}
\end{equation}
for $x\in B_{R_0/\sigma_k}\setminus B_{(\log\sigma_k^{-1})^{1/\delta_0}}$ and $k\gg 1$. Since (\ref{eqn:3-6''}) holds, there exists $C_{\const\insh,\varepsilon_0,\delta_0}>0$ such that 
\begin{align}
&0\leq \tgamk-\tgamk'(x)\leq \frac{1}{2\pi}\int_{B_{2R_0/\sigma_k}\setminus B_{\frac{1}{2}(\log\sigma_k^{-1})^{1/\delta_0}}} \tilde{f}_k \nonumber\\ 
&\quad \leq \frac{1}{2\pi} \cdot C_{\subf,\varepsilon_0} \int_{B_{2R_0/\sigma_k}\setminus B_{\frac{1}{2}(\log\sigma_k^{-1})^{1/\delta_0}}} \vert y\vert^{-(2+3\delta_0)}dy \nonumber\\ 
&\quad \leq C_{\subh,\varepsilon_0,\delta_0}(\log\sigma_k^{-1})^{-3}
 \label{eqn:3-7}
\end{align}
for $x\in B_{R_0/\sigma_k}\setminus B_{(\log\sigma_k^{-1})^{1/\delta_0}}$ and $k\gg 1$. 

Similarly we have 
\begin{align}
&\left\vert \int_{B_{2R_0/\sigma_k}\setminus B_{\vert x\vert/2}} \tilde{f}_k(y)\log\frac{1}{\vert x-y\vert}dy \right\vert \nonumber\\ 
&\quad =\int_{(B_{2R_0/\sigma_k}\setminus B_{\vert x\vert/2})\cap\{\vert y-x\vert\leq 1\}} \tilde{f}_k(y)\log\frac{1}{\vert x-y\vert}dy \nonumber\\
&\quad +\int_{(B_{2R_0/\sigma_k}\setminus B_{\vert x\vert/2})\cap\{\vert y-x\vert>1\}} \tilde{f}_k(y)\log\vert x-y\vert dy \nonumber\\
&\quad \leq C_{\subf,\varepsilon_0}\left\{(\vert x\vert-1)^{-(2+3\delta_0)}\int_{B_1}\log\frac{1}{\vert y\vert}dy \right.\nonumber\\
&\quad \left.+\int_{({\bf R}^2\setminus B_{\vert x\vert/2})\cap\{\vert y-x\vert >1\}}\vert y\vert^{-(2+3\delta_0)}\log\vert x-y\vert dy \right\} \equiv I. 
 \label{eqn:3-8}
\end{align}
Since 
\[ \vert y\vert ^{-\delta_0}\log \vert x-y\vert \leq \vert y\vert^{-\delta_0}\log(\vert x\vert+\vert y\vert) \leq \vert y\vert^{-\delta_0}\log(3\vert y\vert ) \]
for 
\begin{equation} 
y\in({\bf R}^2\setminus B_{\vert x\vert/2})\cap\{\vert y-x\vert>1\}, \ x\in B_{R_0/\sigma_k}\setminus B_{(\log\sigma_k^{-1})^{1/\delta_0}} 
 \label{eqn:3.23}
\end{equation} 
and $k\gg 1$, we have $C_{\const\insj,\delta_0}>0$ such that 
\[ \vert y\vert^{-(2+3\delta_0)}\log\vert x-y\vert\leq C_{\subj,\delta_0}\vert y\vert^{-2(1+\delta_0)} \]
for $(x,y)$ in (\ref{eqn:3.23}) with $k\gg 1$.  Hence we have $C_{\const\insi,\varepsilon_0,\delta_0}>0$ such that
\begin{equation} 
I \leq C_{\subi,\varepsilon_0,\delta_0}(\log\sigma_k^{-1})^{-2}. 
 \label{eqn:3.25}
\end{equation} 

Now we see from (\ref{eqn:3-2'}) and (\ref{eqn:3.21}) that 
\begin{align*}
&\left\vert \twk(x)+\tgamk'(x)\log\vert x\vert\right\vert \leq \frac{1}{2\pi}\int_{B_{2R_0/\sigma_k}}\tilde f_k(y)\left\vert \log \vert y\vert \right\vert dy \\
&\quad + \frac{1}{2\pi}\left\vert \int_{B_{2R_0/\sigma_k}\setminus B_{\vert x\vert/2}}\tilde f_k(y)\log\frac{1}{\vert x-y\vert} dy \right\vert \nonumber\\ 
&\quad +\frac{1}{2\pi}\int_{B_{\vert x\vert/2}}\tilde f_k(y)\left\vert \log \frac{\vert x\vert}{\vert x-y\vert}\right\vert dy +O(1) \\ 
&\quad \leq \frac{1}{2\pi}\int_{B_{2R_0/\sigma_k}}\tilde{f}_k(y) \left\vert \log\vert y\vert \right\vert dy +\frac{\log 2}{2\pi}\|\tilde{f}_k\|_{L^1(B_{2R_0/\sigma_k})} \\
&\quad +\frac{1}{2\pi}\left\vert \int_{B_{2R_0/\sigma_k}\setminus B_{\vert x\vert/2}}\tilde{f}_k(y)\log\frac{1}{\vert x-y\vert}dy \right\vert +O(1)  
\end{align*} 
for $x\in B_{R_0/\sigma_k}\setminus B_{(\log\sigma_k^{-1})^{1/\delta_0}}$.  Therefore, it holds that 
\begin{equation}
\left\vert \twk(x)+\tgamk'(x)\log\vert x\vert\right\vert =O(1) \quad \mbox{as $k\rightarrow\infty$}
 \label{eqn:2.36}
\end{equation} 
by (\ref{eqn:3-6}), (\ref{eqn:3-8}) with (\ref{eqn:3.25}), and the uniform $L^1$ boundedness of $\tilde{f}_k$.  Then (\ref{eqn:3-7}) and (\ref{eqn:2.36}) imply 
\begin{align*}
&\left\vert \twk(x)+\tgamk\log\vert x\vert\right\vert \leq (\tgamk-\tgamk'(x))\log\vert x\vert +\left\vert \twk(x)+\tgamk'(x)\log\vert x\vert\right\vert \\
&\leq C_{\subh,\varepsilon_0,\delta_0}(\log\sigma_k^{-1})^{-3}\log(\sigma_k^{-1}R_0)+O(1) =O(1) \quad \mbox{as $k\rightarrow\infty$}
\end{align*}
for $x\in B_{R_0/\sigma_n}\setminus B_{(\log\sigma_k^{-1})^{1/\delta_0}}$. 
\end{proof} 

Now we complete the proof of Proposition \ref{pro:yyl-type-estimate}. \\

\underline{{\it Proof of Proposition \ref{pro:yyl-type-estimate}}:} \ We take $\delta_0$ and $\tilde{R}_0$ as in (\ref{eqn:3-6'}) and (\ref{eqn:def-tilde-R0}), respectively. 
First, (\ref{eqn:entire-conv}) and (\ref{eqn:3-2'}) with (\ref{eqn:3.9}) imply 
\begin{align}
&\left\vert \twk(x)+\tgamk\log(1+\vert x\vert)\right\vert \leq\vert \tilde w_{k}(x)\vert+\tilde\gamma_k\log (1+\vert x\vert) \nonumber\\ 
&\quad \leq C_{\const\insk}, \quad x\in B_{\tilde{R}_0}, 
 \label{eqn:yyl-1}
\end{align} 
while Lemma \ref{lem:3-3} means 
\begin{equation}
\left\vert \twk(x)+\tgamk\log(1+\vert x\vert)\right\vert \leq C_{\const\insl}, \quad x\in B_{R_0/\sigma_k}\setminus B_{(\log\sigma_k^{-1})^{1/\delta_0}}, 
 \label{eqn:yyl-2}
\end{equation}
where $k\gg 1$.  

Now we put 
\begin{align*}
&\twk^+(x)=-\tgamk\log\vert x\vert+C_{\const\insm}+\frac{C_{\subf,\varepsilon_0}}{9\delta_0^2}\vert x\vert^{-3\delta_0} \\ 
&\twk^-(x)=-\tgamk\log\vert x\vert-C_{\subm}-\frac{1}{4}\vert x\vert^2\tdk\sup_{B_{2R_0}}e^{\xik}
\end{align*}
for $C_{\subm}=1+\max\{C_{\subk},C_{\subl}\}$ and $k\gg 1$, recalling (\ref{eqn:def-tilde-f_n}), and let 
\[ A_k=B_{(\log\sigma_k^{-1})^{1/\delta_0}}\setminus \bar{B}_{\tilde{R}_0}.  \] 
 Then (\ref{eqn:3-6''}) implies  
\begin{align*}
&-\Delta\twk^+=C_{\subf,\varepsilon_0}\vert x\vert^{-(2+3\delta_0)}\geq \tilde{f}_k-\tdk e^{\xik} \quad \mbox{in $A_k$} \\
&\twk^+ \geq \twk  \quad \mbox{on $\partial A_k$}. 
\end{align*}
Next, we have 
\begin{align*}
&-\Delta\twk^-=-\tdk\sup_{B_{2R_0}}e^{\xik} \leq \tilde{f}_k-\tdk e^{\xik} \quad \mbox{in $A_k$} \\
&\twk^- \leq \twk \quad \mbox{on $\partial A_k$}. 
\end{align*} 
Since $-\Delta\twk=\tilde{f}_k-\tdk e^{\xik}$ in $A_k$, it follows from the maximum principle that 
\begin{equation} 
\twk^-\leq \tw\leq\twk^+ \quad \mbox{in $A_k$}. 
 \label{eqn:3.29}
\end{equation} 
Using 
\[ \left\vert \frac{1}{4}\vert x\vert^2\tdk\right\vert \leq C_{\const\insn}, \quad x\in B_{R_0/\sigma_k} \] 
and 
\[ \left|\frac{C_{\subf,\varepsilon_0}}{9\delta_0^2}\vert x\vert^{-3\delta_0}\right|\leq C_{\const\inso}, \quad x\in A_k, \] 
we obtain 
\begin{equation}
\left\vert \twk(x)+\tgamk\log\vert x\vert \right\vert\leq C_{\subm}+\max\{C_{\subn},C_{\subo}\}, \quad x\in A_k 
 \label{eqn:yyl-3}
\end{equation}
for $k\gg 1$. 

Properties (\ref{eqn:yyl-1})-(\ref{eqn:yyl-3}), (\ref{eqn:2.17}) and (\ref{eqn:3.9}) imply (\ref{eqn:1.16}) for $\alpha=1$, 
\[
\wk(x)-\wk(0)=-\left(\frac{4}{\int_{I_+}\beta \calP(d\beta)}+o(1)\right)\log (1+e^{\wk(0)/2}\vert x\vert)+O(1). 
\] 
The other case of $\alpha$ follows from the relation $(\wka(x)-\wka(0))=\alpha(w_{k}(x)-w_{k}(0))$, and the proof is complete. \qed

\section{Proof of Theorem \ref{thm:main}}\label{sec:proof}

We begin with the following lemma. 

\begin{lem} 
It holds that 
\begin{equation} 
\wka(0)=\wk(0)+O(1) \quad \mbox{as $k\rightarrow\infty$} 
 \label{eqn:4.1}
\end{equation}
uniformly in $\alpha\in [\amin,1]$.  
 \label{pro:cor-yyl-type-estimate}
\end{lem}

\begin{proof} \ By the monotonicity (\ref{eqn:monotone-0}), we have only to show 
\begin{equation}
\wkamin(0)=\wk(0)+O(1). 
 \label{eqn:4.1-1}
\end{equation}
As shown in the previous section, estimate (\ref{eqn:1.16}) is equivalent to
\begin{equation}
\wka(x)-\wka(0)=-\alpha\tgamk\log(1+e^{\wk(0)/2}\vert x\vert)+O(1), 
 \label{eqn:4.1-2}
\end{equation}
where $\tgamk$ is as in (\ref{eqn:def-tilgam}). 
Since $\amin\tgamk\geq 2+3\delta_0$ for $k\gg 1$ by (\ref{eqn:3-6'}), we use (\ref{eqn:4.1-2}) to get 
\begin{align}
\int_{B_{R_0}}e^{\wkamin}&=O(1)\cdot e^{\wkamin(0)}\int_{B_{R_0}}(1+e^{\wk(0)/2}|x|)^{-\amin \tgamk}dx \nonumber\\ 
&=O(1)\cdot e^{\wkamin(0)-\wk(0)}\int_{B_{R_0/\sigk}}(1+|x|)^{-\amin \tgamk}dx \nonumber\\ 
&\leq O(1)\cdot e^{\wkamin(0)-\wk(0)}\int_{{\bf R}^2}(1+|x|)^{-(2+3\delta_0)}dx. 
 \label{eqn:4.1.-3}
\end{align}
If (\ref{eqn:4.1-1}) fails then (\ref{eqn:4.1.-3}) and Lemma \ref{lem:3-1} imply that $\wkamin\rightarrow -\infty$ uniformly in $\Omega\setminus B_{R_0/2}$, 
and therefore we conclude $\int_\Omega e^{\wkamin}\rightarrow 0$ as $k\rightarrow\infty$, which contradicts $\int_\Omega e^{\wkamin}=1$. 
\end{proof}

\begin{lem} 
It holds that 
\begin{equation} 
\limsup_{k\rightarrow\infty}\int_{I_+} (\bar{w}_{k,\alpha}+\wka(0))\calP(d\alpha)>-\infty. 
 \label{eqn:4.2}
\end{equation} 
\end{lem} 

\begin{proof} \ Fix $x'\in\partial B_{R_0/2}$. 
Then it holds that 
\begin{align}
\bar{w}_{k,\alpha}&=w_{k,\alpha}(x')+O(1)=\wka(0)-\alpha\tgamk\log(1+e^{\wk(0)/2}|x'|)+O(1) \nonumber\\
&=\left(1-\frac{\alpha\tgamk}{2}\right)w_k(0)+O(1)
 \label{eqn:mean-wka}
\end{align}
by (\ref{eqn:3.1})-(\ref{eqn:3.2}), (\ref{eqn:4.1-2}) and (\ref{eqn:4.1}). 
Since $\tgam\geq \tgamk$ and $\int_{I_+}(1-\alpha\tgam/4)\calP(d\alpha)=0$ by (\ref{eqn:gam-estimate-above}) and (\ref{eqn:2.17}), respectively, it follows that  
\begin{align*}
&\int_{I_+} (\bar{w}_{k,\alpha}+\wka(0))\calP(d\alpha)= 2\wk(0)\int_{I_+}\left(1-\frac{\alpha\tgamk}{4}\right)\calP(d\alpha)+O(1) \\
&\quad =2w_k(0)\int_{I_+}\left(1-\frac{\alpha\tgam}{4}\right)\calP(d\alpha)+\frac{1}{2}w_k(0)(\tgam-\tgamk)\int_{I_+}\alpha\calP(d\alpha)+O(1) \\
&\quad \geq 2w_k(0)\int_{I_+}\left(1-\frac{\alpha\tgam}{4}\right)\calP(d\alpha)+O(1)=O(1), 
\end{align*}
where we have used (\ref{eqn:mean-wka}) and (\ref{eqn:4.1}) in the first equality. 
\end{proof}

\begin{lem}
It holds that 
\begin{equation}
J_{\lambda_k}(v_k)=\frac{\lambda_k}{2}\int_{I_+}\left(\bar{w}_{k,\alpha}+\int_\Omega \wka e^{\wka}\right)\calP(d\alpha). 
 \label{eqn:4-0}
\end{equation}
 \label{lem:4}
\end{lem}

\begin{proof} By (\ref{eqn:1.13h}) and $\int_\Omega v_k=0$, we have 
\begin{equation}
\frac{1}{2}\int_\Omega \vert \nabla v_k\vert^2=\frac{1}{2\alpha^2}\int_\Omega \vert \nabla \wka\vert^2 
 \label{eqn:4-1}
\end{equation}
and
\begin{equation}
\bar{w}_{k,\alpha}=-\log\left(\int_\Omega e^{\alpha v_k} dx\right), 
 \label{eqn:4-2}
\end{equation}
respectively. 
Multiplying (\ref{eqn:eqnw}) by $\wka$ and using 
\[
\int_\Omega v_k=0, \quad \wka=\alpha v_k+\bar{w}_{k,\alpha}, \quad \int_\Omega e^{\wkb}=1, 
\]
we get 
\begin{align}
&\int_\Omega \vert \nabla \wka\vert^2=\alpha\lambda_k\int_\Omega \wka\left(\int_{I_+}\beta\left(e^{\wkb}-\frac{1}{|\Omega|}\right)\calP(d\beta)\right) dx \nonumber\\ 
&\quad =\alpha^2\lambda_k\int_\Omega v_k \int_{I_+}\beta e^{\wkb}\calP(d\beta) \ dx \nonumber\\
&\quad =\alpha^2\lambda_k\int_{I_+}\int_\Omega \wkb e^{\wkb}dx \ \calP(d\beta)-\alpha^2\lambda_k\int_{I_+}\bar{w}_{k,\beta}\calP(d\beta). 
 \label{eqn:4-3}
\end{align}

We combine (\ref{eqn:4-1})-(\ref{eqn:4-3}) with $\calP(I_+)=1$ to obtain  
\begin{align*}
&J_{\lambda_k}(v_k)=\frac{1}{2}\int_\Omega\vert \nabla v_k\vert^2-\lambda_k\int_{I_+}\log\left(\int_\Omega e^{\alpha v_k}\right)\calP(d\alpha) \\
&\quad =\frac{1}{2}\int_{I_+}\frac{1}{\alpha^2}\int_\Omega\vert \nabla \wka\vert^2 dx \ \calP(d\alpha)
+\lambda_k \int_{I_+}\bar{w}_{k,\alpha}\calP(d\alpha) \\
&\quad =\frac{\lambda_k}{2}\int_{I_+}\int_\Omega \wka e^{\wka} dx \ \calP(d\alpha)-\frac{\lambda_{k}}{2}\int_{I_+}\bar{w}_{k,\alpha}\calP(d\alpha) \\ 
&\quad +\lambda_k \int_{I_+}\bar{w}_{k,\alpha}\calP(d\alpha) =\frac{\lambda_k}{2}\int_{I_+}\left(\bar{w}_{k,\alpha}+\int_\Omega w_{k, \alpha}e^{\wka}\right)\calP(d\alpha). 
\end{align*}
The proof is complete. 
\end{proof} 

We now prove Theorem \ref{thm:main} in the following. \\ 

\underline{{\it Proof of Theorem \ref{thm:main}:}} \ We shall show that (\ref{eqn:1.9h}) holds. To this end we apply $\int_\Omega e^{\wka}=1$ in (\ref{eqn:4-0}) and get 
\begin{align*} 
J_{\lambda_k}(v_k)&=\frac{\lambda_k}{2}\left\{\int_{I_+}(\bar{w}_{k,\alpha}+\wka(0))\calP(d\alpha) \right. \\ 
&\quad \left. +\int_{I_+}\calP(d\alpha)\int_\Omega (\wka(x)-\wka(0))e^{\wka(x)}dx\right\}. 
\end{align*} 
Hence the proof of (\ref{eqn:1.9h}) is reduced to showing 
\begin{equation}
\int_{I_+}\calP(d\alpha)\int_\Omega (\wka(x)-\wka(0))e^{\wka(x)}dx=O(1), 
 \label{eqn:main-2}
\end{equation}
thanks to (\ref{eqn:4.2}). 

To show (\ref{eqn:main-2}), we take $x'\in\Psi_k^{-1}(\partial B_{R_0/2})$. 
Then, (\ref{eqn:3.1})-(\ref{eqn:3.2}), (\ref{eqn:4.1-2}), (\ref{eqn:4.1}), (\ref{eqn:1.12h}) and (\ref{eqn:3.9}) imply, 
uniformly in $x\in \Omega\setminus \Psi_k^{-1}(B_{R_0})$ and $\alpha\in [\amin,1]$, that 
\begin{align*}
&(\wka(x)-\wka(0))e^{\wka(x)}=e^{O(1)}(O(1)+\wka(x')-\wka(0))e^{\wka(x')} \\ 
&=e^{O(1)} \left(O(1)-\alpha\tgamk\log \left( 1+e^{\wk(0)/2}\vert x'\vert\right)\right) e^{\wka(0)-\alpha\tgamk\log \left(1+e^{\wk(0)/2}\vert x'\vert\right)} \nonumber\\ 
&=e^{O(1)} \left(O(1)-\frac{\alpha\tgamk}{2}w_k(0) \right) e^{-\left(\frac{\alpha\tgamk}{2}-1\right)\wk(0)}=o(1). 
\end{align*}
Hence it follows that 
\begin{equation}
\int_{[\amin,1]}\calP(d\alpha)\int_{\Omega\setminus \Psi_k^{-1}(B_{R_0})}(\wka(x)-\wka(0))e^{\wka(x)}dx=o(1). 
 \label{eqn:main-2'}
\end{equation}
Finally, (\ref{eqn:4.1-2}), (\ref{eqn:4.1}), (\ref{eqn:1.12h}) and (\ref{eqn:3.9}) imply 
\begin{align}
&\int_{B_{R_0}}(\wka(x)-\wka(0))e^{\wka(x)+\xik(x)}dx \nonumber\\ 
&=-e^{\wka(0)+O(1)}\int_{B_{R_0}}e^{(\wka(x)-\wka(0))+\xik(x)}\log \left(1+e^{\wk(0)/2}\vert x\vert\right)dx+O(1) \nonumber\\ 
&=-e^{\wka(0)+O(1)}\int_{B_{R_0}}\frac{\log(1+e^{\wk(0)/2}\vert x\vert)}{\left(1+e^{\wk(0)/2}\vert x\vert\right)^{\alpha\tgamk}}dx+O(1) \nonumber\\ 
&=-e^{\wka(0)-\wk(0)+O(1)}\int_{B_{e^{\wk(0)/2}R_0}}\frac{\log(1+\vert x\vert)}{\left(1+\vert x\vert\right)^{\alpha\tilde{\gamma}_k}}dx+O(1)
=O(1)
 \label{eqn:main-2''}
\end{align}
uniformly in $\alpha\in [\amin,1]$. Then (\ref{eqn:main-2}) follows from (\ref{eqn:main-2'}) and (\ref{eqn:main-2''}).  \qed\\ 

\appendix

\section{Proof of Lemma \ref{lem:2-a}}\label{sec:appendixA}
Given $K>0$, we put 
\begin{align*}
&I_1(x)=\int_{D_1}\frac{\log|x-y|-\log(1+|y|)-\log|x|}{\log|x|}f(y)dy \\
&I_{2,K}(x)=\int_{D_{2,K}}\frac{\log|x-y|-\log(1+|y|)-\log|x|}{\log|x|}f(y)dy \\
&I_{3,K}(x)=\int_{D_{3,K}}\frac{\log|x-y|-\log(1+|y|)-\log|x|}{\log|x|}f(y)dy, 
\end{align*}
where 
\begin{align*}
&D_1=D_1(x)\equiv \{y\in{\bf R}^2 \mid \vert y-x\vert\leq 1\} \\
&D_{2,K}=D_{2,K}(x)\equiv \{y\in{\bf R}^2 \mid \vert y-x\vert>1, \ \vert y\vert\leq K\} \\
&D_{3,K}=D_{3,K}(x)\equiv \{y\in{\bf R}^2 \mid \vert y-x\vert>1, \ \vert y\vert> K\}.
\end{align*}
Then it holds that 
\[ \frac{z(x)}{\log \vert x\vert}-\gamma=\frac{1}{2\pi}(I_1(x)+I_{2,K}(x)+I_{3,K}(x)). \]
We have only to show that each $\varepsilon>0$ admits $K_\varepsilon$ and $L_\varepsilon>0$ such that 
\begin{equation}
|I_1(x)|+|I_{2,K_\varepsilon}(x)|+|I_{3,K_\varepsilon}(x)|\leq\varepsilon 
 \label{eqn:lem:2-2-1-1}
\end{equation}
for all $x\in{\bf R}^2\setminus B_{L_\varepsilon}$. 

Since 
\[ \frac{\log(1+|y|)+\log|x|}{\log|x|}\leq\frac{\log(2+|x|)+\log|x|}{\log|x|}\leq 3, \quad x\in{\bf R}^2\setminus B_2, \ y\in D_1(x), \] 
we have 
\begin{align}
|I_1(x)| 
&\leq 3\int_{D_1(x)}f(y)dy-\frac{1}{\log|x|}\int_{D_1(x)}f(y)\log|x-y|dy \nonumber\\
&\leq 3\int_{D_1(x)}f(y)dy-\frac{\|f\|_\infty}{\log|x|}\int_{B_1}\log|y|dy \rightarrow 0
 \label{eqn:lem:2-2-1-2}
\end{align}
uniformly as $|x|\rightarrow +\infty$, recalling $f\in L^1\cap L^\infty({\bf R}^2)$. 

Next, we have 
\[ \left|\frac{\log|x-y|-\log(1+|y|)-\log|x|}{\log|x|}\right|
\leq\frac{1}{\log|x|}\left\{\log(1+K)+\left|\log\frac{|x-y|}{|x|}\right|\right\} \]
for $x\in{\bf R}^2\setminus B_2$ and $y\in D_{2,K}(x)$, and thus
\begin{equation}
|I_{2,K}(x)|\leq\frac{1}{\log|x|}\int_{D_{2,K}(x)}\left\{\log(1+K)+\left|\log\frac{|x-y|}{|x|}\right|\right\}f(y)dy  
 \label{eqn:lem:2-2-1-2'}
\end{equation}
for $x\in{\bf R}^2\setminus B_2$. 
From  
\[ \frac{1}{2+|x|}\leq \frac{|x-y|}{1+|y|}\leq 1+|x|, \quad x\in{\bf R}^2, \ \vert y-x\vert\geq 1, \] 
we derive
\[ \left|\frac{\log|x-y|-\log(1+|y|)-\log|x|}{\log|x|}\right|\leq 3, \quad x\in{\bf R}^2\setminus B_2, \ |y-x|\geq 1 \] 
to obtain 
\begin{equation}
|I_{3,K}(x)|\leq 3\int_{D_{3,K}(x)}f(y)dy \leq 3\int_{{\bf R}^2\setminus B_K}f(y)dy
 \label{eqn:lem:2-2-1-2''}
\end{equation}
for $x\in{\bf R}^2\setminus B_2$.

Recalling $0\leq f\in L^1({\bf R}^2)$, let $\varepsilon_0>0$ be given. From (\ref{eqn:lem:2-2-1-2''}), there exists $K_0>0$ such that 
\[ |I_{3,K}(x)|\leq \varepsilon_0 \]
for all $K\geq K_0$ and $x\in{\bf R}^2\setminus B_2$. Next, by (\ref{eqn:lem:2-2-1-2'}) any $K>0$ admits $L_K>0$ such that 
\[ |I_{2,K}(x)|\leq \varepsilon_0 \]
for all $x\in{\bf R}^{2}\setminus B_{L_K}$, and therefore 
\begin{equation}
|I_{2,K_0}(x)|+|I_{3,K_0}(x)|\leq 2\varepsilon_0
 \label{eqn:lem:2-2-1-3}
\end{equation}
for all $x\in{\bf R}^{2}\setminus B_{L_{K_0}}$. 

Thus we obtain (\ref{eqn:lem:2-2-1-1}) by (\ref{eqn:lem:2-2-1-2}) and (\ref{eqn:lem:2-2-1-3}). \qed


\begin{thebibliography}{99}

\bibitem{aubin} 
Aubin T.: 
Some Nonlinear Problems in Riemannian Geometry. 
Springer-Verlag, Berlin (1998) 

\bibitem{bm91} 
Brezis, H., Merle, F.: 
Uniform estimates and blowup behavior for solutions of $-\Delta u = V(x)e^u$ in two dimensions. 
Comm. Partial Differential Equations {\bf 16}, 1223--1253 (1991) 

\bibitem{bls93} 
Brezis, H., Li, Y.Y., Shafrir, I.: 
A $\sup+\inf$ inequality for some nonlinear elliptic equations involving exponential nonlinearities. 
J. Funct. Anal. {\bf 115}, 344--358 (1993) 



\bibitem{ck94} 
Chanillo, S., Kiessling, M.: 
Rotational symmetry of solutions of some nonlinear problems in statistical mechanics and in geometry. 
Comm. Math. Phys. {\bf 160}, 217--238 (1994) 

\bibitem{cl93} 
Chen, W., Li, C.: 
Qualitative properties of solutions to some nonlinear elliptic equations in ${\bf R}^2$. 
Duke Math. J. {\bf 71}, 427--439 (1993) 

\bibitem{csw97} 
Chipot, M., Shafrir, I., Wolansky, G.: 
On the solutions of Liouville systems. 
J. Differential Equations {\bf 140}, 59--105 (1997) 

\bibitem{ew09} 
Esposito, P., Wei, J.: 
Non-simple blow-up solutions for the Neumann two-dimensional $\sinh$-Gordon equation. 
Calc. Var. {\bf 34}, 341--375 (2009) 

\bibitem{es06} 
Eyink, G.L., Sreenivasan, K.R.: 
Onsager and the theory of hydrodynamic turbulence. 
Reviews of Modern Physics {\bf 78}, 87--135 (2006) 

\bibitem{jm73} 
Joyce, G., Montgomery, D.: 
Negative temperature states for the two-dimensional guiding-centre plasma. 
J. Plasma Phys. {\bf 10}, 107--121 (1973) 

\bibitem{jwyz08} 
Jost, J., Wang, G., Ye, D., Zhou, C.: 
The blow up analysis of solutions of the elliptic $\sinh$-Gordon equation. 
Calc. Var. {\bf 31}, 263--276 (2008) 


\bibitem{yyl99} 
Li, Y.Y.: 
Harnack type inequality: the method of moving planes. 
Comm. Math. Phys. {\bf 200}, 421--444 (1999) 

\bibitem{ls94} 
Li, Y.Y., Shafrir, I.: 
Blow-up analysis for solutions of $-\Delta u = Ve^u$ in dimension two. 
Indiana Univ. Math. J. {\bf 43}, 1255--1270 (1994) 

\bibitem{lieb-loss} 
Lieb, E.H., Loss, M.: 
Analysis, second edition. 
American Mathematical Society, Providence, RI (2001)

\bibitem{csl07} 
Lin, C.-S.: 
An expository survey of the recent development of mean field equations. 
Discrete Contin. Dyn. Syst. {\bf 19}, 387--410 (2007) 

\bibitem{ors10} 
Ohtsuka, H., Ricciardi, T., Suzuki, T.: 
Blow-up analysis for an elliptic equation describing stationary vortex flows with variable intensities in 2D-turbulence. 
J. Differential Equations {\bf 249}, 1436--1465 (2010) 

\bibitem{os06} 
Ohtsuka, H., Suzuki, T.: 
Mean field equation for the equilibrium turbulence and a related functional inequality. 
Adv. Differential Equations {\bf 11}, 281--304 (2006)  

\bibitem{onsager} 
Onsager, L.: 
Statistical hydrodynamics. 
Nuovo Cimento Suppl. {\bf 26}, 279--287 (1949) 

\bibitem{pl76} 
Pointin, Y.B., Lundgren, T.S.: 
Statistical mechanics of two-dimensional vortices in a bounded container. 
Phys. Fluids {\bf 19}, 1459--1470 (1976) 

\bibitem{pw} 
Protter, M.H., Weinberger, H.F.: 
Maximum Principles in Differential Equations. 
Springer-Verlag, New York (1984)

\bibitem{rs-pre} 
Ricciardi T., Suzuki, T.: 
Duality and best constant for a Trudinger-Moser inequality involving probability measures. 
J. Eur. Math. Soc. (JEMS) {\bf 16}, 1327--1348 (2014)  

\bibitem{sawada-suzuki08} 
Sawada, K., Suzuki, T.: 
Derivation of the equilibrium mean field equations of point vortex and vortex filament system. 
Theoret. Appl. Mech. Japan {\bf 56}, 285--290 (2008) 

\bibitem{shafrir92} 
Shafrir, I.: 
Une in{\'e}galit{\'e} de type ${\rm Sup}+{\rm Inf}$ pour l'{\'e}quation $-\Delta u=Ve^u$. 
C. R. Acad. Sci. Paris {\bf 315}, 159--164 (1992) 

\bibitem{sw05} 
Shafrir, I., Wolansky, G.: 
The logarithmic HLS inequality for systems on compact manifolds. 
J. Funct. Anal. {\bf 227}, 200--226 (2005) 

\bibitem{Suzuki07} 
Suzuki, T.: 
Mean Field Theories and Dual Variation. 
Atlantis Press, Amsterdam-Paris (2008) 

\bibitem{sz-rims} 
Suzuki, T., Zhang, X.: 
Trudinger-Moser inequality for point vortex mean field limit with multi-intensities. 
RIMS Kokyuroku {\bf 1837}, 1--19 (2013) 

\bibitem{Tarantello08} 
Tarantello, G.: 
Selfdual Gauge Field Vortices. An Analytical Approach. 
Birkh{\" a}user Boston (2008)

\end{thebibliography}
\end{document}